\documentclass[11pt,a4paper]{article}
\usepackage{graphicx}
\usepackage{amsfonts}
\usepackage{amssymb}
\usepackage{amsthm}
\usepackage[centertags]{amsmath}
\usepackage{newlfont}

\usepackage{color}
\usepackage{graphicx,color}
\usepackage{amsmath, amssymb, graphics}
 
\def\r{\mathbb R}

\newtheorem{theorem}{Theorem}[section]
\newtheorem{corollary}[theorem]{Corollary}
\newtheorem{proposition}[theorem]{Proposition}
\newtheorem{definition}[theorem]{Definition}

\title{Classification of rotational surfaces in Euclidean space satisfying a linear relation between their principal curvatures}
\author{Rafael L\'opez\footnote{Partially
supported by MEC-FEDER
 grant no. MTM2017-89677-P}\\
 Departamento de Geometr\'{\i}a y Topolog\'{\i}a\\ Instituto de Matem\'aticas (IEMath-GR)\\
 Universidad de Granada\\
 18071 Granada, Spain\\
\texttt{rcamino@ugr.es}\\
\\
\'Alvaro P\'ampano\footnote{Partially supported by MINECO-FEDER grant MTM2014-54804-P and Gobierno Vasco grant
IT1094-16. The author has been supported by Programa Predoctoral del Gobierno Vasco, 2015}\\
Departament of  Mathematics\\ Faculty of Science and Technology\\
University of the Basque Country\\
48940 Bilbao, Spain\\
\texttt{alvaro.pampano@ehu.eus}}
\date{}
\begin{document}
\maketitle

\begin{abstract}
We classify all rotational surfaces in Euclidean space whose principal curvatures $\kappa_1$ and $\kappa_2$ satisfy the linear relation $\kappa_1=a\kappa_2+b$, where $a$ and $b$ are two constants. We give a variational characterization of these surfaces in terms of its generating curve. As a consequence of our classification,  we find closed (embedded and not embedded) surfaces  and periodic (embedded and not embedded) surfaces with a geometric behaviour similar to Delaunay surfaces.
\end{abstract}
\section{Introduction and summary of shapes} 

In this paper we investigate surfaces in    the Euclidean three-dimensional space $\r^3$ satisfying the linear relation $a\kappa_1+b\kappa_2=c$ between the principal curvatures $\kappa_1$ and $\kappa_2$,   where $a, b$ and $c$ are three real constants, such that $a^2+b^2\neq 0$. From now, we will discard the trivial case where the two constants $a$ and $c$ (resp. $b$ and $c$) are both $0$, because in such case, $\kappa_2=0$ (resp. $\kappa_1=0$) and, then the surface is developable and trivially  satisfies the above linear relation. Following Chern \cite{ch}, a Weingarten surface  is a surface  where   $\kappa_1$ and $\kappa_2$ satisfy a certain relation $\Phi(\kappa_1,\kappa_2)=0$. These surfaces were introduced by Weingarten in \cite{we} and its study occupies an important role in classical differential geometry.   A first result due to Chern proves that the sphere is the only ovaloid with the property that $\kappa_1$ is a decreasing function of $\kappa_2$ \cite{ch} (for example, if $a<0$).   Later, Hopf   proved in \cite{lo} that there do not exist closed analytic surfaces of genus greater or equal than $2$ unless $a=-1$, that is, the surface has constant mean curvature and if the genus is $0$ and the surface is analytic and rotational, then $a$ or $1/a$ must be an odd integer.  Indeed, for each $a>1$, Hopf  proved the existence of  a non-spherical closed convex rotational  $C^2$-surface. The particular case  $a=-1$  is exceptional. Hopf proved that the sphere is the only closed surface of genus $0$ with constant mean curvature (\cite{ho2}). During many years, it was conjectured that the sphere was the only closed surface with constant mean curvature until   in 1986 Wente found an immersed torus in $\r^3$ with constant mean curvature \cite{wente}. Later, Kapouleas proved the existence of closed surfaces for arbitrary genus \cite{ka}.

We point out that surfaces satisfying a relation $\Psi(H,K)=0$ between the mean curvature $H$ and the Gauss curvature $K$ have been considered in the literature: here we only  refer \cite{cft,gmm,rs}. However  the linear case $pH+qK=r$ is equivalent to $p\kappa_1+p\kappa_2+2q\kappa_1\kappa_2=2r$, which is not of type $a\kappa_1+b\kappa_2=c$.

Other surfaces satisfying the relation  $a\kappa_1+b\kappa_2=c$ are those ones where    one of the principal curvatures is constant. If for example  $\kappa_1$ is constant, we take $b=0$, and for each $a\in\r$, $c=a\kappa_1$ satisfies (\ref{rw}). From the above paragraph, we exclude the case $c=0$, that is, that $\kappa_1=0$. Surfaces with one constant principal curvature were classified in \cite{st} and they are spheres or tubes along a regular curve. Because in this paper we are concerned with rotational surfaces, then the only surfaces are spheres and tori of revolution. 

After these examples, we rewrite the linear relation $a\kappa_1+b\kappa_2=c$ and we  give the next definition.

\begin{definition} A linear Weingarten surface $M$ in $\r^3$ is a surface such that 
\begin{equation}\label{rw}
\kappa_1(p)=a\kappa_2(p)+b,\quad p\in M,
\end{equation}
where $a,b\in\r$, $a\not=0$.
\end{definition}

 Examples of linear Weingarten surfaces are the following:
\begin{enumerate}
\item \emph{Umbilical surfaces}. This is the case  when  $a=1$ and $b=0$. Then  $M$ is a part of a plane or a part of a round sphere.
\item \emph{Isoparametric surfaces}. In this case   both principal curvatures are constant. Besides the umbilical surfaces, the surface must be a  circular cylinder. 
\item \emph{Constant mean curvature surfaces}. This is the case when  $a=-1$ and the surface  has constant mean curvature $H=b/2$. 
\end{enumerate}

From the  above results, it is clear that the class of rotational linear Weingarten surfaces deserves to be known explicitly. However, and surprisingly,  up today these surfaces are not completely classified and this is one of the main objectives of this paper. A second purpose is to give  a variational characterization of the profile curve of this class of surfaces, which is 
also unexpected because variational methods are commonly associated to the concepts of the mean curvature $H$ or the Gauss curvature $K$. 

We now review the main results of the rotational linear Weingarten surfaces. In \cite{pa}, the author only computed   the differential equation of the generating curve when $b=0$. Recall that Hopf proved in \cite{ho} the existence of  convex closed rotational surfaces for any $a>0$: see also \cite{ks} for the existence of rotational closed surfaces with other relations $\Phi(\kappa_1,\kappa_2)$. When $a=2$ and $b=0$, Mladenov and Oprea have named this surface as the Mylar balloon \cite{mo1}. If $a>0$ and $b=0$, they have also given   parametrizations of the closed surfaces in terms of elliptic and hypergeometric functions and show that the surface is a   critical point  of a variational problem \cite{mo2}. In the special case $b=0$, Barros and Garay proved that all the parallels of these rotational surfaces are critical points for an energy functional involving the normal curvature and acting on the space of closed curves immersed in the surface \cite{bg}: see also some graphics  in \cite{ku}.  On the other hand, the first author studied linear Weingarten surfaces foliated by a uniparametric group of circles, proving that the surface is rotational or   the surface is one of the minimal examples of Riemann \cite{lo}.

In many of the works mentioned above, the authors only study closed rotational surfaces, specially ovaloids. For instance, if $b=0$, and not assuming rotational symmetry, then $K=a\kappa_2^2$ and thus, if $a<0$, there do not exist closed surfaces. Recall that by the Hilbert  lemma,  the round sphere is  the only ovaloid satisfying (\ref{rw}) with   $a<0$.  When $a>0$, $K=a\kappa_1^2+b\kappa_1$ and the surface may have points where the curvature is negative.   Here we have in mind the constant mean curvature equation, that is, $a=-1$ in (\ref{rw}), because many of  the  surfaces that will appear in our classification of Section \ref{sec5} share similar properties with the rotational surfaces with constant mean curvature. These surfaces were characterized in 1841 by Delaunay as surfaces    generated by   a roulette of a conic foci along the rolling axis \cite{de}. They are planes and  catenoids (minimal case), spheres, unduloids and nodoids. Unduloids and nodoids are periodic surfaces along the axis where unduloids  are embedded while nodoids are not. Unduloids may be viewed as smooth deformations of the cylinder, and the transition between unduloids and nodoids occurs through spheres. 

{\it Convention:} Along this paper,  the principal curvatures of a rotational surface  are going to be denoted by $\kappa_1$ and $\kappa_2$, where $\kappa_1$ will be the curvature of the profile curve $\gamma$, while $\kappa_2$ denotes the normal curvature of the orbit of the rotation. 

One of the main goal of our paper is a variational characterization of the generating curve of a rotational linear Weingarten surface. We will prove that this curve is a critical point of an energy functional involving a power of the curvature $\kappa$ of the curve. 

As a consequence of our classification, we will obtain a complete description of the  rotational linear Weingarten surfaces  that we now summarize: see Figures \ref{fig-cla1}, \ref{fig-cla2} and \ref{fig-cla3}. Denote by $\gamma$ the generating curve.   If the surface does not meet the axis of rotation,  we give the next definitions:
\begin{enumerate}
\item {\it Catenoid-type} surfaces. The curve $\gamma$ is  a concave graph on some interval $I$ of the axis. These surfaces only appear when $a<0$ and $b=0$. There are two types depending if $I=\r$   ($-1\leq a<0$) or if $I$ is a bounded interval  ($a<-1$). The plane is  included here as an extremal case.
\item  {\it Unduloid-type} surfaces. Embedded surfaces which are periodic in the direction of the axis. Circular cylinders belong to this family.
\item  {\it Nodoid-type} surfaces. Non embedded surfaces which are periodic   in the direction of the axis and the curve $\gamma$ has loops towards the axis.   
\item  {\it Antinodoid-type} surfaces. Non embedded  surfaces which are periodic in the direction of the axis and the curve $\gamma$ has loops facing away from the axis.   
\item {\it Cylindrical antinodoid-type} surfaces. Non embedded surfaces asymptotic to a circular cylinder. The curve $\gamma$ has a single loop facing away from the axis.
\end{enumerate}

Now, we turn to those surfaces that meet (necessarily orthogonally)  the axis of rotation. All the surfaces have genus $0$ except in one case that the surface touches the axis at exactly one point.
\begin{enumerate}
\item {\it Ovaloids}. They are convex surfaces. The shape is like an oblate spheroid being more flat close to the axis as the parameter $a$ gets bigger. This case only occurs  when $a>0$.  Round spheres are included here.
\item {\it Vesicle-type} surfaces. Embedded closed surfaces where the two poles of the profile curve are close so the meridian presents two inflection points. These surfaces have concave regions around the poles. 
\item  {\it Pinched spheroids}. Limit case of the vesicle-type surfaces when the two poles coincide. The surface is tangentially immersed on the axis and bounds a solid three-dimensional torus. 
\item {\it Immersed spheroids}. Closed surfaces of genus $0$ that appear when the two poles of the vesicle-type surface pass their-self  through the axis. 
\end{enumerate}

This paper is organized as follows.  In Section \ref{sec2} we give a variational characterization of the generating curve of rotational surfaces verifying \eqref{rw}. Notice that this characterization is completely different from that one given in \cite{bg}, since the involved variational problems have nothing in common. Indeed, here, contrary to \cite{bg}, the extremal curves are going to be the meridians. Then, in Section \ref{sec3}, we show some properties about   symmetries of the  solutions of \eqref{rw}.  In Section \ref{sec4},   we consider the  case  $b=0$ in \eqref{rw}. Finally, in   section  Section \ref{sec5}, we give the classification when $b\neq0$. For this purpose, we distinguish between two cases, when the parameter $a$ in \eqref{rw} is positive or negative.

\begin{figure}[hbtp]
\begin{center}
\includegraphics[width=.3\textwidth]{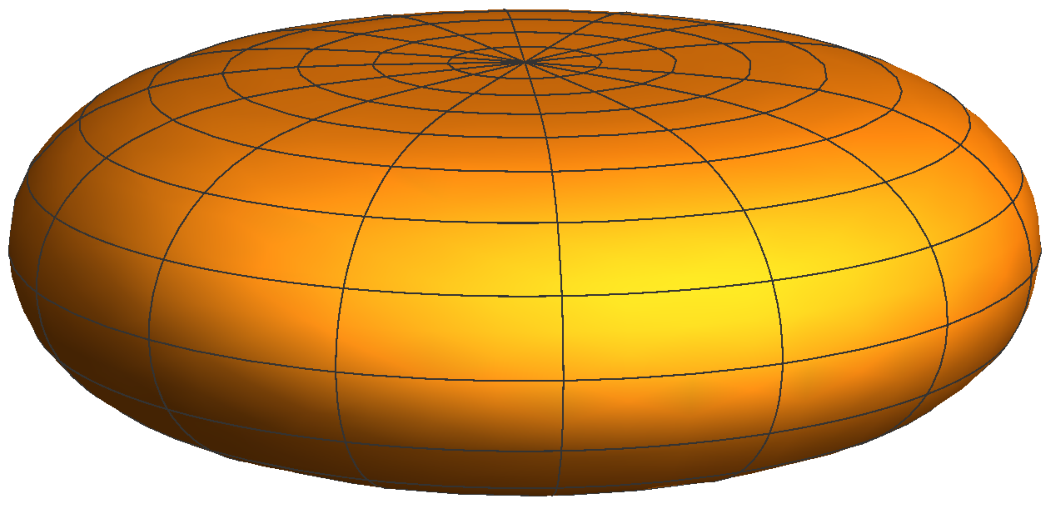}\quad \includegraphics[width=.3\textwidth]{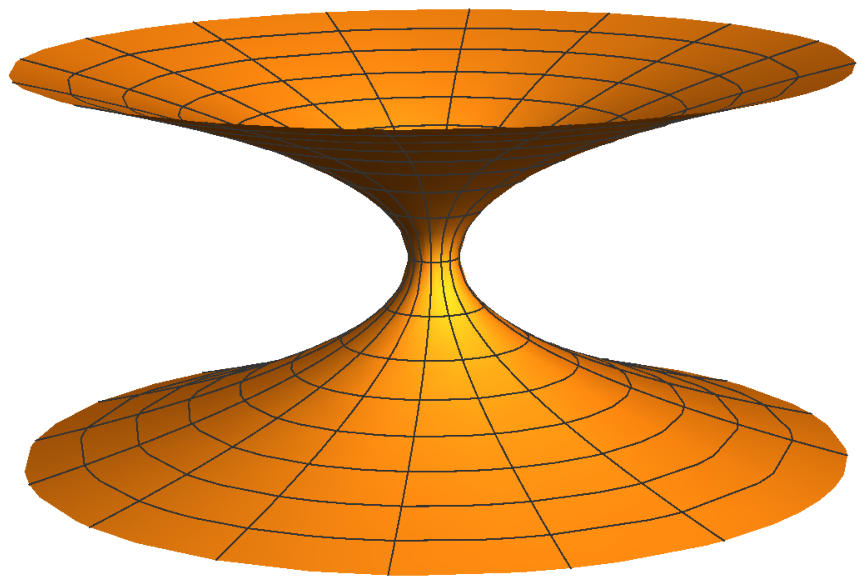}\quad \includegraphics[width=.3\textwidth]{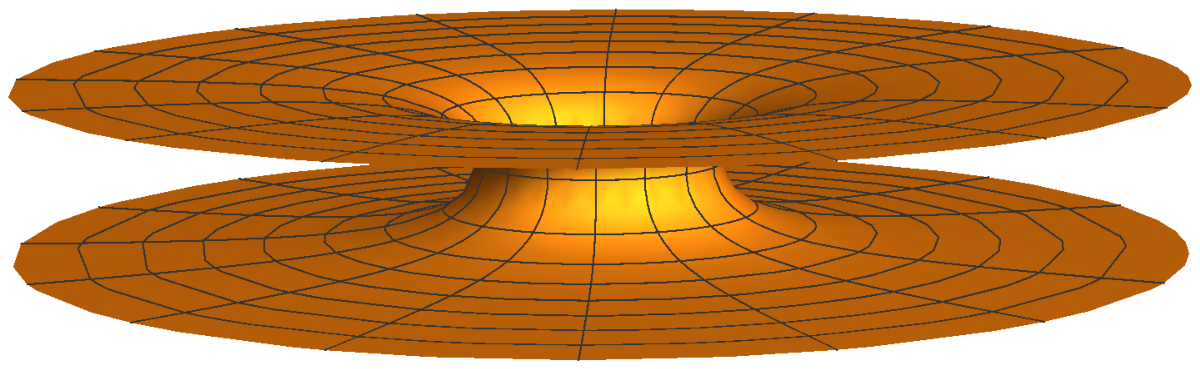}
\end{center}
\caption{Rotational surfaces satisfying the relation $\kappa_1=a\kappa_2$. From left to right: ovaloid ($a>0$),  catenoid-type with $a\in [-1,0)$ and catenoid-type with $a<-1$}\label{fig-cla1}
\end{figure}

\begin{figure}[hbtp]
\begin{center}
\includegraphics[width=.13\textwidth]{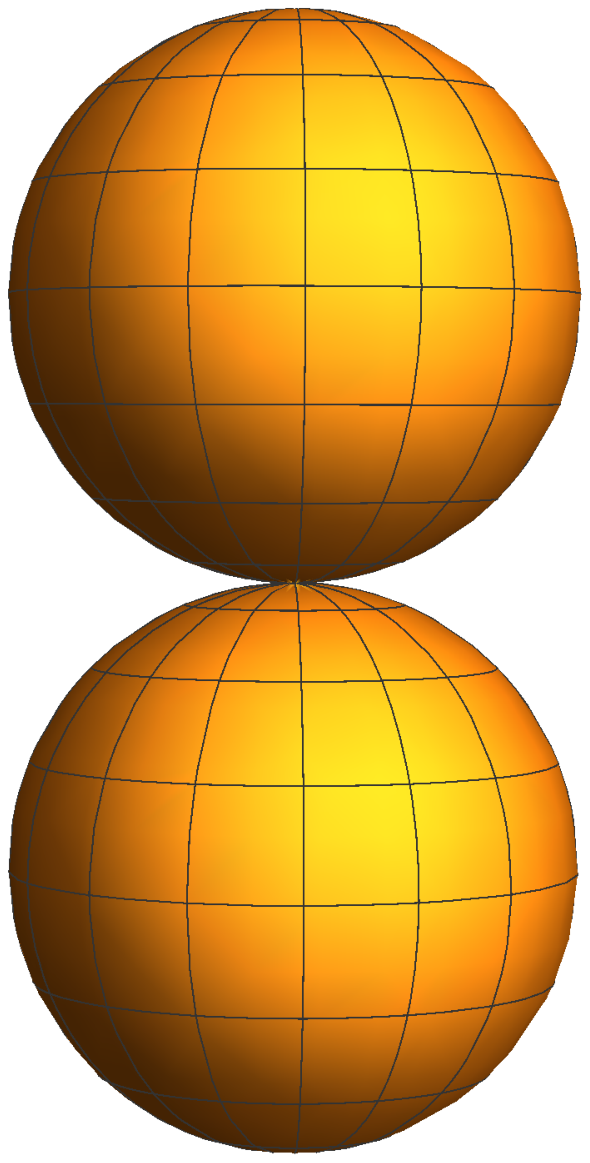}\quad \includegraphics[width=.35\textwidth]{figureb012.eps}\quad\includegraphics[width=.4\textwidth]{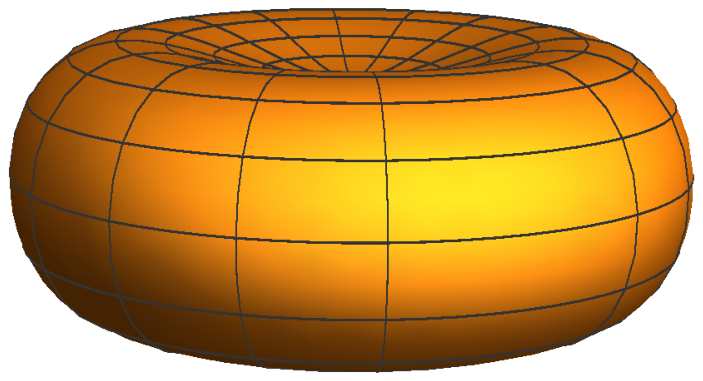}  \includegraphics[width=.2\textwidth]{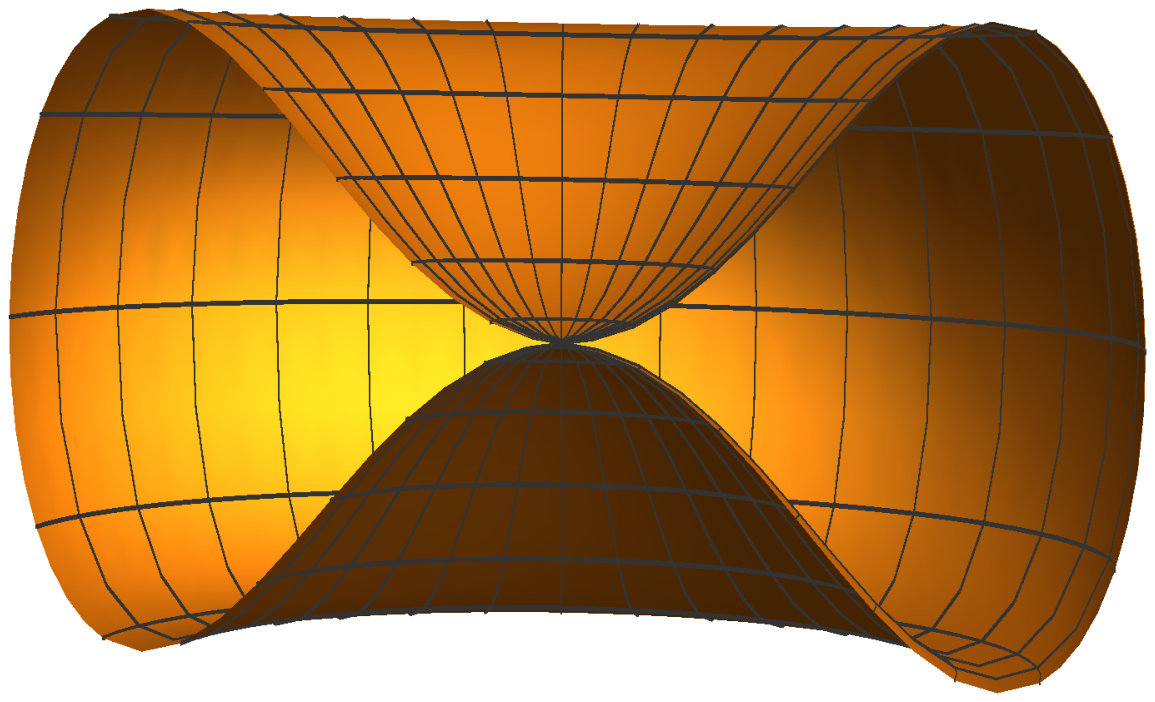}  \includegraphics[width=.2\textwidth]{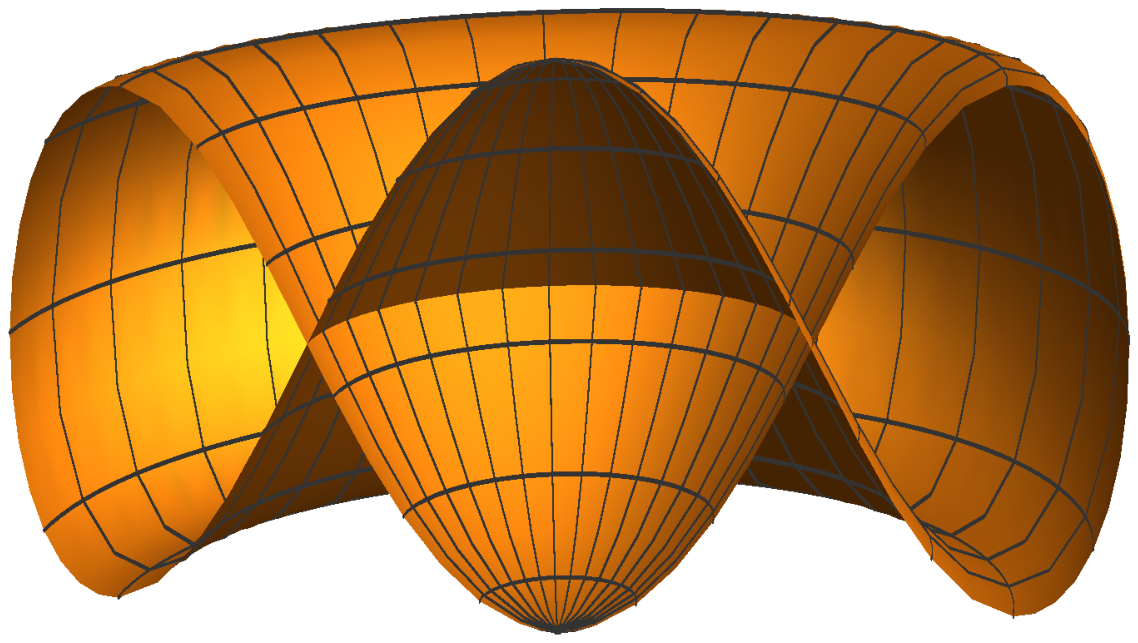} \includegraphics[width=.1\textwidth]{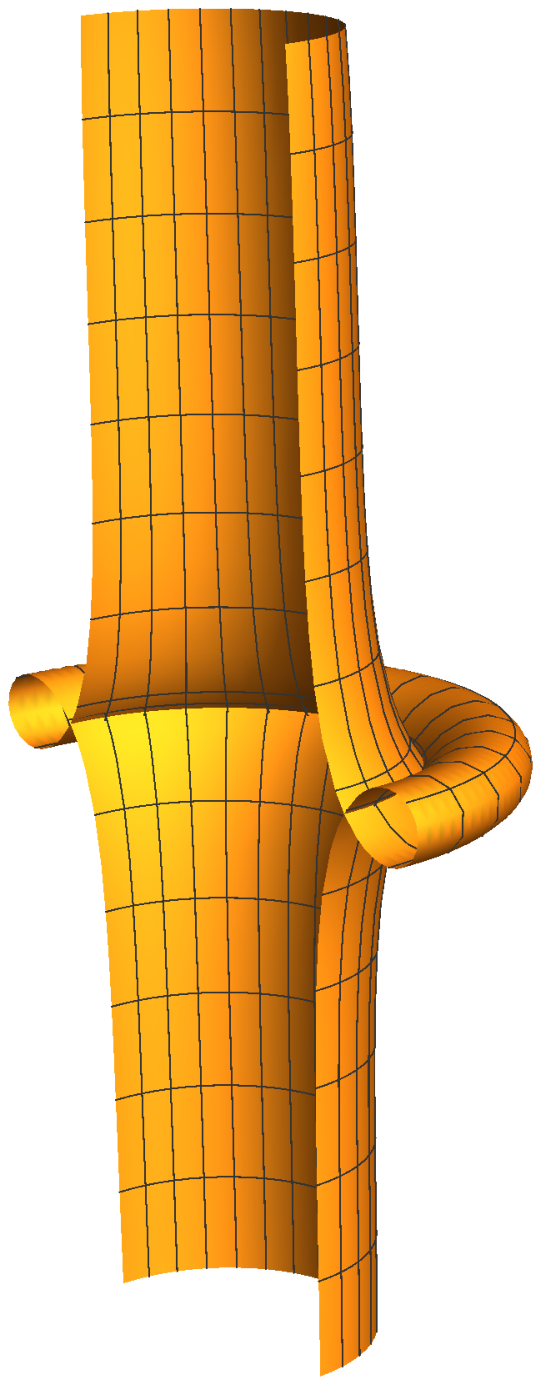} \includegraphics[width=.1\textwidth]{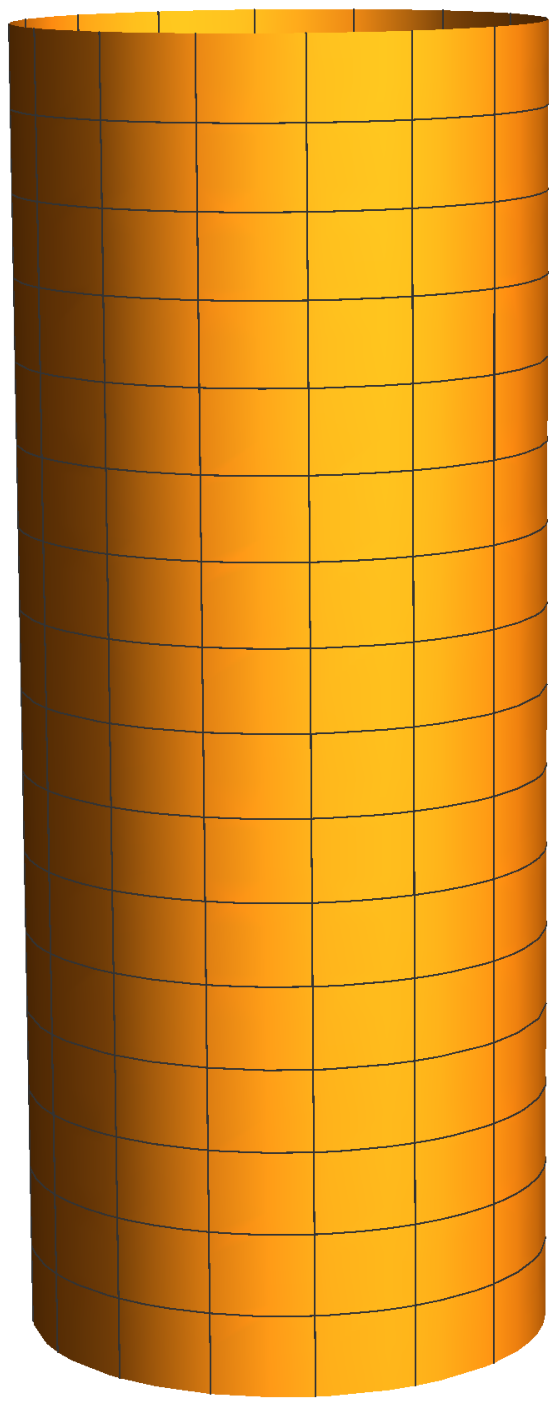} \includegraphics[width=.15\textwidth]{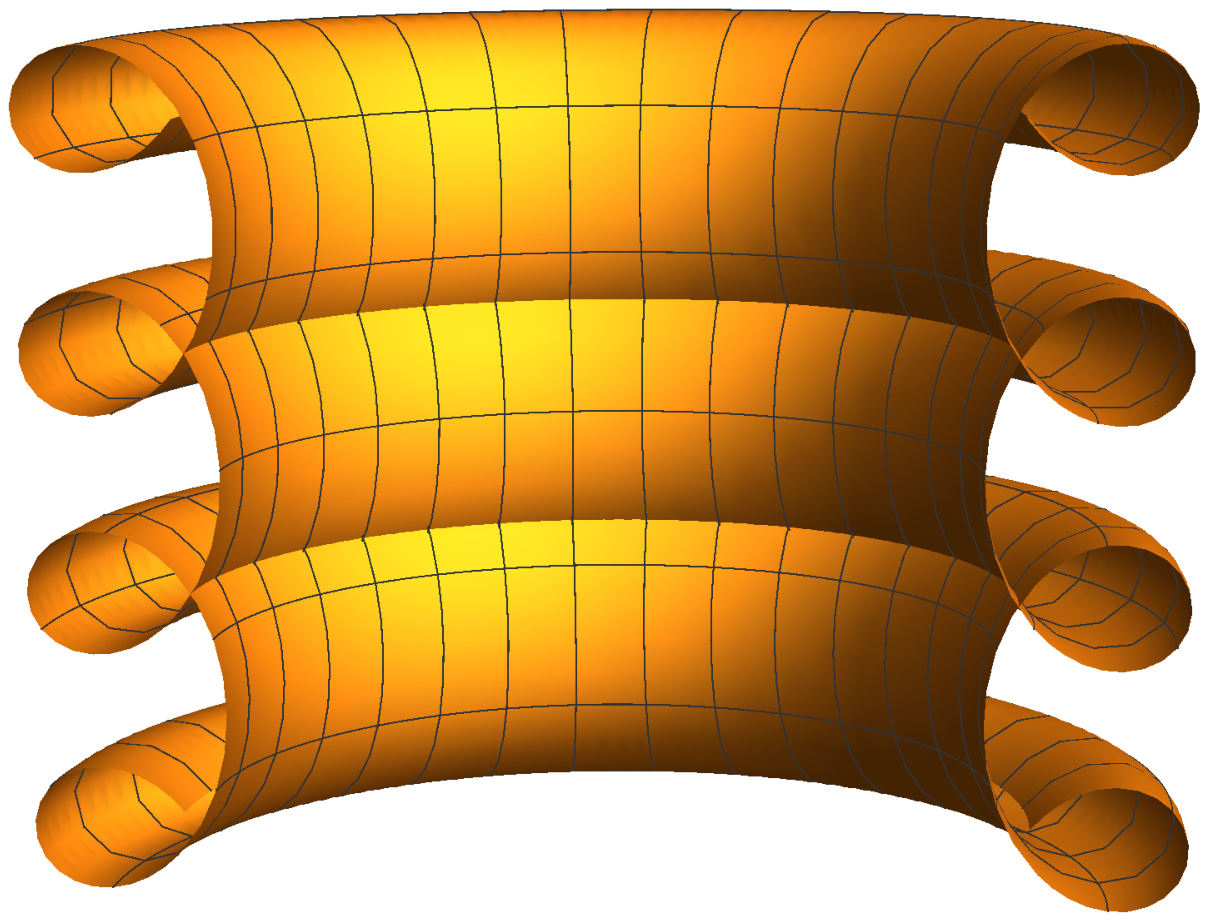}
\end{center}
\caption{Rotational surfaces satisfying the relation $\kappa_1=a\kappa_2+b$ with $a>0$. From left to right: sphere, ovaloid, vesicle-type, pinched spheroid, immersed spheroid, cylindrical antinodoid-type, circular cylinder, antinodoid-type}\label{fig-cla2}
\end{figure}

\begin{figure}[hbtp]
\begin{center}
\includegraphics[width=.2\textwidth]{figureb1.eps}  \includegraphics[width=.15\textwidth]{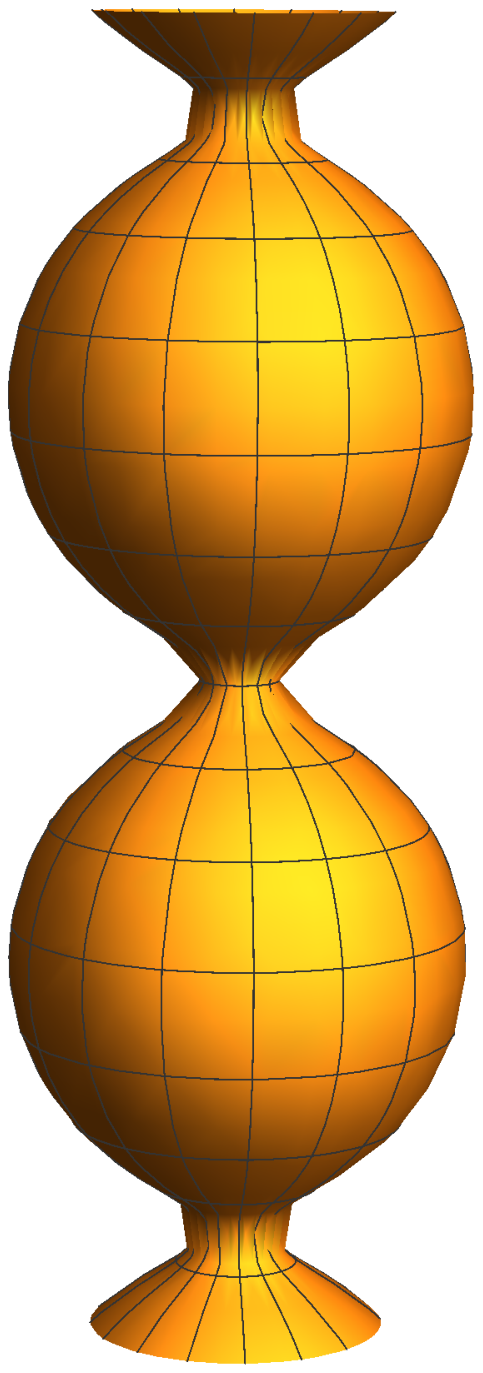}  \includegraphics[width=.15\textwidth]{figureb3.eps}  \includegraphics[width=.3\textwidth]{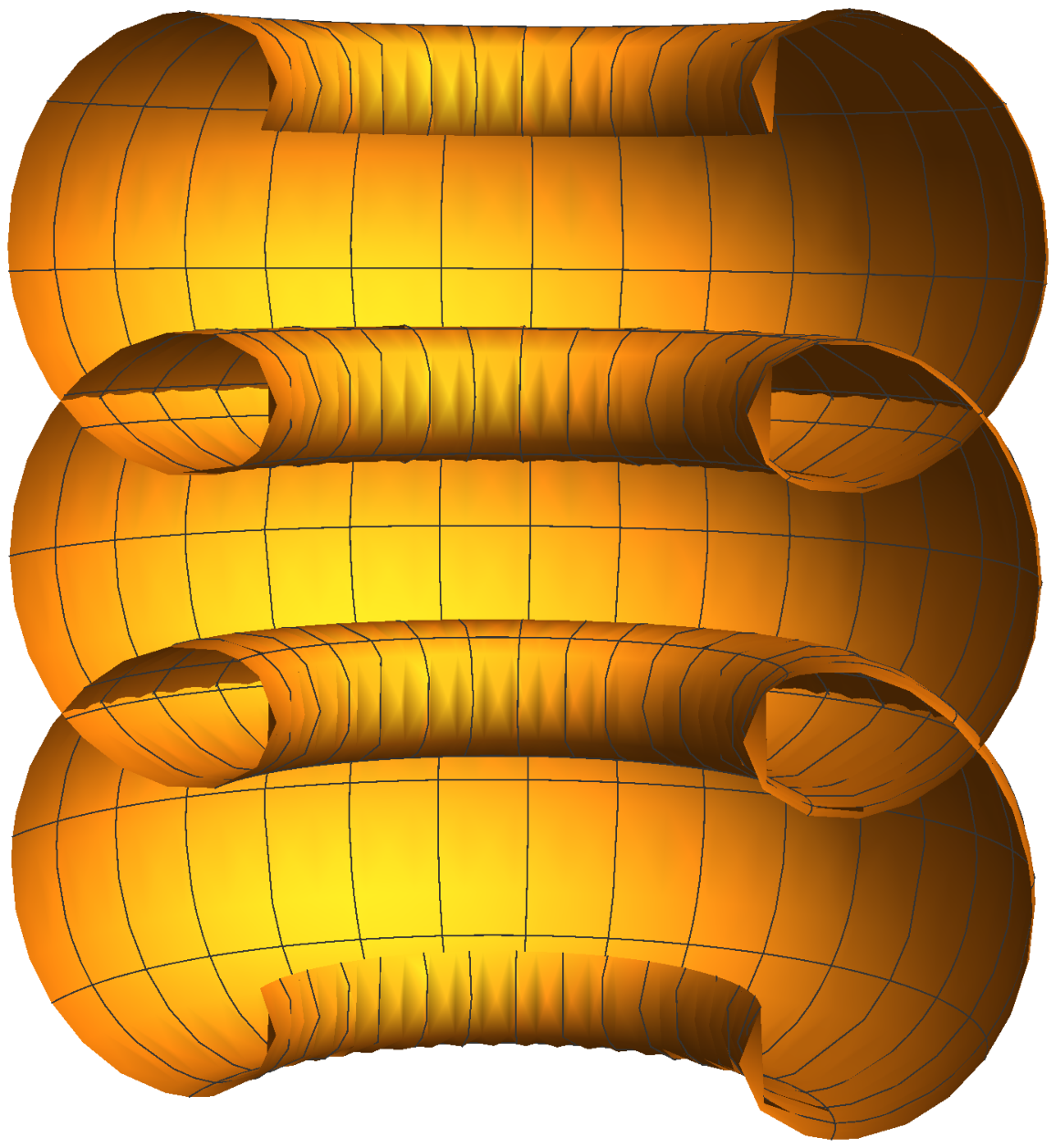}
\end{center}
\caption{Rotational surfaces satisfying the relation $\kappa_1=a\kappa_2+b$ with $a<0$. From left to right: sphere, unduloid-type, circular cylinder, nodoid-type}\label{fig-cla3}
\end{figure}

\section{Variational characterization of generating curves}\label{sec2}

Let $(x,y,z)$ be the canonical coordinates in the Euclidean space $\r^3$ and let $M\subset\r^3$ be a surface of revolution. Without loss of generality, we   assume that the rotational axis is the $z$-axis and that   its generating curve $\gamma$ is contained in the $xz$-plane. Let  $\gamma(s)=(x(s),0,z(s))$, $s\in I\subset\r$, be  parametrized by the arc-length, and thus $x'(s)=\cos\theta(s)$ and $z'(s)=\sin\theta(s)$ for a certain function $\theta$.  Recall our convention that $\kappa_1$ is the curvature of the profile curve. If $X(s,\phi)=(x(s)\cos\phi,x(s)\sin\phi,z(s))$ is a parametrization of $M$, then the principal curvatures are
$$\kappa_1(s)=\theta'(s),\quad\kappa_2(s)=\frac{\sin\theta(s)}{x(s)},\quad s\in I.$$
Notice that they are independent of the rotation angle $\phi$. Therefore, a surface of revolution $M$ satisfying the linear relation (\ref{rw}) is characterized by the following system of ordinary differential equations
\begin{equation}\label{eq1}
	\begin{split}
		x'(s)&=\cos\theta(s)\\
		z'(s)&=\sin\theta(s)\\
		\theta'(s)&=a\frac{\sin\theta(s)}{x(s)}+b.
	\end{split}
\end{equation}
In this  section  we   characterize variationally the   curve $\gamma(s)$. Let us denote $\Omega_{p_0p_1}$   the space of smooth regular curves in $\mathbb{R}^2$ joining two fixed points $p_0$ and $p_1$ of $\mathbb{R}^2$. Let  $\Omega_{p_0p_1}^*$ be the subspace of  those curves of $\Omega_{p_0p_1}$ satisfying $| \gamma ''|>\mu$, that is, 
\begin{equation}
	\Omega_{p_0p_1}^*=\{\beta:\left[a_0,a_1\right]\rightarrow\mathbb{R}^2: \beta(a_i)=p_i, i\in\{0,1\}, \frac{d\beta}{dt}(t)\neq 0, \forall t\in \left[a_0,a_1\right], \langle\beta'',\beta''\rangle^{\frac{1}{2}}>\mu\}\, . \nonumber
\end{equation}

For a curve $\gamma:\left[a_0,a_1\right]\rightarrow\mathbb{R}^2$ we take a variation of $\gamma$, $\Gamma=\Gamma(t,\bar{t}):\left[a_0,a_1\right]\times(-\varepsilon,\varepsilon)\rightarrow\mathbb{R}^2$ with $\Gamma(t,0)=\gamma(t)$. Associated to this variation, we have the vector field $\widetilde{W}=\widetilde{W}(t)=\frac{\partial\Gamma}{\partial t}(t,0)$ along the curve $\gamma(t)$. Moreover, if $W$ is any proper vector field along a curve $\gamma\in\Omega_{p_0p_1}$, then it is known that there exists a variation of $\gamma$ by immersed curves in $\mathbb{R}^2$, $\Gamma:\left[a_0,a_1\right]\times(-\varepsilon,\varepsilon)\rightarrow\mathbb{R}^2$, $(t,\bar{t})\rightarrow\Gamma(t,\bar{t})$, whose variation vector field is $W$. Indeed, if $\gamma\in\Omega_{p_0p_1}^*$, smoothness of $\gamma$ and $| \gamma ''|$ implies that there exists a sub-variation $\hat{\Gamma}$ of $\Gamma$ with the same variation vector field $W$, such that any variation curve in $\hat{\Gamma}$ belongs to $\Omega_{p_0p_1}^*$.

For each $p,\mu\in\r$, define   the curvature energy functional
\begin{equation}
	\mathcal{F}^{\mu *}_p(\gamma)=\int_o^L\left(\theta'(s)-\mu\right)^p\, ds\,, \label{fpmu*}
\end{equation}
acting on $\Omega_{p_0p_1}^*$. This  functional has been studied in \cite{gp} where their correspondence Euler-Lagrange equations have been related to solutions of a generalized Ermakov-Milne-Pinney ordinary differential equation. When $p=0$,  \eqref{fpmu*} is nothing but the length functional whose critical curves are geodesics; and if $p=1$, then \eqref{fpmu*} is, basically, the total curvature functional in which case extremals are the planar curves (see \cite{gp} and references therein). From now on, we discard these two cases, so   $p\neq0$, $1$.

In this section, we also  need to define  the energy
\begin{equation}
	\mathcal{F}_\nu(\gamma)=\int_o^L\, e^{\nu\theta'(s)}\, ds\, \label{fnu}
\end{equation}
among curves immersed in $\Omega_{p_0p_1}$ and where $\nu\in\mathbb{R}-\{0\}$. 

If $\gamma(s)$ is a geodesic of $\mathbb{R}^2$, that is, if $\theta'=0$, then it is clear that $\gamma(s)$ is a global extremal curve of \eqref{fpmu*} and \eqref{fnu}, provided they act on a space of $L^1$ integrable curves whenever it makes sense. Therefore, if $M$ is a plane or a circular cylinder, then its profile curve is an extremal curve of either \eqref{fpmu*} or \eqref{fnu}. In fact, this result can be generalized to all rotational linear Weingarten surfaces. Let us assume $\gamma(s)$ is not a geodesic, then we obtain the main theorem of this section.

\begin{theorem}\label{1direction} Let $M$ be a rotational linear Weingarten surface  and let $\gamma(s)=\left(x(s),0,z(s)\right)$ be its generating curve. Then,
	\begin{enumerate}
		\item If $a\neq 1$, $\gamma$ is an extremal curve (under arbitrary boundary conditions) of $\mathcal{F}_p^{\mu *}$ for
		\begin{equation} \mu=\frac{-b}{a-1}\,,\quad p=\frac{a}{a-1}\,.\nonumber\end{equation}
		\item If $a=1$ and $b\neq 0$, then $\gamma$ satisfies the Euler-Lagrange equation of $\mathcal{F}_\nu$ for
		\begin{equation} \nu=\frac{1}{b}\,.\nonumber\end{equation}
	\end{enumerate}
\end{theorem}
\begin{proof}
	The first step of the proof consists on computing the Euler-Lagrange equations associated to \eqref{fpmu*} and to \eqref{fnu} when acting on $\Omega_{p_0p_1}^*$ and $\Omega_{p_0p_1}$, respectively. For this purpose, consider $\gamma: I\rightarrow \mathbb{R}^2$ a regular immersed curve  joining $p_0$ and $p_1$ is an extremal curve of $\mathcal{F}^{\mu *}_p$. Then, if $W$ is a proper vector field along $\gamma$, that is, an infinitesimal variation of the curve, we find 
	\begin{equation}
		\partial\mathcal{F}_p^{\mu *}(W)=\frac{\partial}{\partial\varepsilon}_{\lvert_{\varepsilon=0}}\mathcal{F}_p^{\mu *}(\gamma+\varepsilon W)=0\,,\nonumber
	\end{equation}
	that is, after reparametrizing the curves of the variation so that all of them have the same fixed domain $\left[a,b\right]$,
	\begin{eqnarray}
		0&=&\frac{\partial}{\partial\varepsilon}_{\lvert_{\varepsilon=0}}\int_0^{L(w)}\left(\lvert \left(\gamma+\varepsilon W\right)''\rvert-\mu\right)^p\, ds=\nonumber \\&=&\frac{\partial}{\partial \varepsilon}_{\lvert_{\varepsilon=0}}\int_{a_o}^{a_1}\left( \lvert\left( \gamma+\varepsilon W\right)''\rvert-\mu\right)^p\, \lvert \left(\gamma+\varepsilon W\right)'\rvert \, dt=\nonumber\\&=&\int_0^{L(w)}(\theta'(s)-\mu)^{p-1}\left(\langle \frac{p}{\theta'(s)}\gamma'',W''\rangle+\langle\left((1-2p)\theta'(s)-\mu\right)\gamma',W'\rangle\right)\, ds\,.\nonumber
	\end{eqnarray}
	It follows, after integrating by parts twice,
		\begin{equation}
		0=\int_0^L\langle\mathcal{E}(\gamma),W\rangle\, ds+\mathcal{B}(W,\gamma)\lvert_0^L\, ,\nonumber
	\end{equation}
	where the Euler-Lagrange operator is
	\begin{equation}
		\mathcal{E}(\gamma)=\left(\frac{p}{\theta'(s)}\frac{d}{ds}\left( (\theta'(s)-\mu)^{p-1}\right)\gamma''+(\theta'(s)-\mu)^{p-1}\left((p-1)\theta'(s)+\mu\right)\gamma'\right)'\, ,\nonumber
	\end{equation}
 where the boundary term $\mathcal{B}(W, \gamma)$ vanishes under suitable boundary conditions. Thus, as $\gamma$ is a critical curve  under any boundary conditions, it follows  by standard arguments   that $\mathcal{E}(\gamma)=0$, that is, the Euler-Lagrange equation of   $\mathcal{F}_p^{\mu *}$ acting on $\Omega_{p_0p_1}^*$ is 	\begin{equation}
		\frac{d^2}{ds^2}\left(\left(\theta'(s)-\mu\right)^{p-1}\right)+\theta'(s)^2\left(\theta'(s)-\mu\right)^{p-1}-\frac{\theta'(s)}{p}\left(\theta'(s)-\mu\right)^p=0\, . \label{el1}
	\end{equation}
Similarly, the Euler-Lagrange equation of  $\mathcal{F}_{\nu}$ acting on $\Omega_{p_0p_1}$ is	\begin{equation}
		\frac{d^2}{ds^2}\left( e^{\nu\theta'(s)}\right)+\theta'(s)^2 e^{\nu\theta'(s)}-\frac{\theta'(s)}{\nu} e^{\nu\theta'(s)}=0\, . \label{el2}
	\end{equation}
	Now, for the second and last step, suppose that $\gamma(s)=(x(s),0,z(s))$ satisfies \eqref{rw}. The first equation of \eqref{eq1} becomes
	\begin{equation}
	\left(\frac{x''(s)}{\theta'(s)}\right)'+x'(s)\theta'(s)=0\,, \label{nec}
	\end{equation}
	while, the last equation of \eqref{eq1} is
	\begin{equation}
	\theta'(s)=-a\frac{x''(s)}{x(s)\theta'(s)}+b\,.\label{nec1}
	\end{equation}
	Assume first that $\theta'(s)$ is constant. If $a\neq 1$, then  the last equation of \eqref{eq1} gives 
	\begin{equation}
	\theta'(s)=\theta'_0=\frac{-b}{a-1}=\mu\,.\nonumber
	\end{equation}
	This case represents a global minimum of \eqref{fpmu*}   acting on a space of curves verifying $\left(\theta'(s)-\mu\right)^p\in L^1(\left[a_0,a_1\right])$. Now, if $a=1$, by \eqref{eq1} it follows that $b=0$, which is out of our consideration.
	On the other hand, if $\theta'(s)$ is not constant,   by the inverse function theorem, we can suppose that $s$ is a function of $\theta'$, and therefore, $x(s)=x(\theta')=\dot{F}(\theta')$ for some smooth function $F$:  here the derivative with respect to $\theta'$ is denoted by  the upper dot. Then \eqref{nec} can be integrated obtaining
	\begin{equation}
	\dot{F}_{ss}+\theta'(s)^2\dot{F}-\theta'(s)\left( F+\lambda\right)=0\, ,\label{nec2}
	\end{equation}
	for a constant $\lambda$.
	Furthermore, combining \eqref{nec2} with the last equation of \eqref{eq1}, \eqref{nec1}, and $x(s)=\dot{F}(\theta')$ we find that
	\begin{equation}
	\dot{F}\left((a-1)\theta'+b\right)=a(F+\lambda)\,. \label{int}
	\end{equation}
	Consequently, if $a=1$, 
	\begin{equation}
	F(\theta')=e^{\frac{1}{b}\theta'}-\lambda\,,\label{Fnu}
	\end{equation}
	which implies that Equation \eqref{nec2} boils down to the Euler-Lagrange equation \eqref{el2} for $\nu=1/b$. 
	Now, if $a\neq 1$, then 
	\begin{equation}
	F(\theta')=\left(\theta'+\frac{b}{a-1}\right)^{\frac{a}{a-1}}-\lambda\,,\label{Fpmu}
	\end{equation}
	and \eqref{nec2} is precisely \eqref{el1} for $\mu= -b/(a-1)$ and $p=a/(a-1)$. This finishes the proof.
\end{proof}

Moreover, the converse of Theorem \ref{1direction} is also true. Indeed, suppose $\gamma$ is a critical curve with   constant curvature. If $\gamma(s)$ is a straight line, then $\gamma$  generates a plane, a cone or a right cylinder. Observe that the cone is the only one which does not satisfy \eqref{rw}. If $\gamma(s)$ is a circle, then $M$ is either a sphere or a tori of revolution. The sphere satisfies \eqref{rw} whereas it is easy to check   that the torus of revolution is not a  linear Weingarten surface. 

On the other hand, from previous proof we derive that $x(s)=\dot{F}(\theta'(s))$ for certain functions $\dot{F}$, \eqref{Fnu} and \eqref{Fpmu}. Therefore, any critical curve of $\mathcal{F}_p^{\mu *}$ in the $xz$-plane with non-constant curvature can be parametrized, up to rigid motions, as
\begin{equation}
	\gamma(s)=d\left(p\left(\theta'-\mu\right)^{p-1} ,0 ,\int \left(\theta'-\mu\right)^{p-1}\left(\left(p-1\right)\theta'+\mu\right)ds \right)\, , \label{critical1}
\end{equation}
for some positive constant $d$. Using the Euler-Lagrange equation \eqref{el1}, it is easy to check that the rotational surface generated by $\gamma$   satisfies the relation \eqref{rw} between its principal curvatures. 

Similarly, up to rigid motions, we can parametrize any extremal curve $\gamma$ of $\mathcal{F}_\nu$  in the $xz$-plane as
\begin{equation}
	\gamma(s)=d\left(\nu e^{\nu\theta'} ,0 ,\int\left(\nu\theta'-1\right)e^{\nu\theta'}ds \right)\, , \label{critical2}
\end{equation}
where, again $d$ is any positive constant.

Then, arguing as before, we conclude with the converse of Theorem  \ref{1direction}. We sum up this result in the following proposition.

\begin{proposition}\label{converse} Let $\gamma$ denote a curve in $\mathbb{R}^2$ with non-constant curvature. If $\gamma$ is a critical curve of $\mathcal{F}_p^{\mu *}$, then $\gamma$  can be parametrized by \eqref{critical1}, up to rigid motions. Similarly, if $\gamma$ is critical of    $\mathcal{F}_\nu$, then $\gamma$  can be parametrized by \eqref{critical2}, again up to rigid motions. Moreover, in both cases, the rotational surface generated by rotating the critical curve $\gamma$ around the $z$-axis satisfies the Weingarten relation \eqref{rw}, where 
\begin{equation}
a=\frac{p}{p-1}\,, \quad b=\frac{-\mu}{p-1}\,, \nonumber
\end{equation}
if the functional is $\mathcal{F}_p^{\mu *}$, or
\begin{equation}
a=1\,,\quad b=\frac{1}{\nu}\,,\nonumber
\end{equation}
if $\gamma$ is critical for $\mathcal{F}_\nu$.
\end{proposition}

As a consequence   of this variational characterization, we prove that, except spheres,  there are not closed rotational surface  for the pure linear case, that is, for $b=0$ in \eqref{rw} (see also Theorem \ref{t-w1}). Let  $M$ be a closed rotational linear Weingarten surface. First, notice that if $a=1$ and $b=0$, then $M$ must be a totally umbilical surface. Therefore, we assume now $a\neq 1$, and then, from our variational characterization (Theorem \ref{1direction} and Proposition  \ref{converse}), its generating curve $\gamma$ is critical for \eqref{fpmu*}. If the critical curve $\gamma$ has constant curvature, then as mentioned above, it generates either a plane, a cylinder or a sphere. Thus, up to here, the only closed surfaces is the sphere, which has genus 0.

From now on, we are going to consider that $\gamma$ has non-constant curvature. In order $M$ to be closed, we need either $\gamma$ to be closed or that it cuts the axis of rotation. In the latter, $M$ cannot be a torus. That is, in order to look for rotational linear Weingarten tori for $b=0$, we must look for closed critical curves of \eqref{fpmu*}. 

\begin{proposition}\label{noclosed} There are no closed critical curves with non-constant curvature of $\mathcal{F}_p^{\mu *}$ for $\mu=0$. As a consequence, the sphere is the only closed rotational surface satisfying $\kappa_1=a\kappa_2$. 
\end{proposition}

\begin{proof}

Critical curves with non-constant curvature of $\mathcal{F}_p^{\mu *}$ can be parametrized by \eqref{critical1}. Therefore, $\gamma$ will be closed  if and only if  $\theta'(s)$ is a periodic function and
\begin{equation}
\int_0^T \left(\theta'(s)-\mu\right)^{p-1}\left((p-1)\theta'(s)-\mu\right)\,ds=0\,,\label{cond}
\end{equation}
where $T$ denotes the period of $\theta'(s)$. If  $\mu=0$,  \eqref{cond} simplifies to
\begin{equation}
\int_0^T\theta'(s)^p\, ds=0\,.\label{cond2}
\end{equation}
Since  the orientation can be locally fixed, say  $\theta'(s)>0$, we obtain a contradiction. 

For the last part of the statement, if   $a=1$, we know that $M$ is a sphere. Thus, we suppose $a\not=1$ and that $\gamma$ has not constant curvature: if it is constant, then $M$ is a plane, a circular cylinder or a sphere. If $b=0$ in (\ref{rw}), then $\mu=0$ and the result is proved.
\end{proof}

\section{Results on symmetry}\label{sec3}

In this section we will obtain some symmetry results on the shape of a rotational linear Weingarten surface. They will be a direct consequence of the uniqueness of the theory of ordinary differential equations.  Following with the notation of the above section, a surface of revolution $M$ satisfying the linear relation (\ref{rw}) is characterized by the system of ordinary differential equations \eqref{eq1}. We now express the initial conditions for (\ref{eq1}), that is, the initial point $\gamma(0)=(x(0),0,z(0))$ and the initial velocity $\gamma'(0)=(x'(0),0,z'(0))$ at $s=0$. The last condition is equivalent to give the initial value $\theta(0)$ for the angle function $\theta(s)$. Since any Euclidean translation in the $z$-direction leaves invariant (\ref{eq1}) and the rotational axis, we may suppose $z(0)=0$. Therefore, the  initial conditions are
\begin{equation}\label{eq2}
x(0)=x_0,\ z(0)=0,\  \theta(0)=\theta_0, \end{equation}
with $x_0>0$ and $\theta_0\in [0,2\pi)$. Consequently the classification of   the linear Weingarten rotational surfaces can be expressed as follows:
\begin{quote} Classify and give a geometrical description of the generating curves $\gamma(s)=(x(s),0,z(s))$ that are solutions of (\ref{eq1}) for any values $x_0>0$ and $\theta_0\in [0,2\pi)$.
\end{quote}

A first result is how equation (\ref{rw}) changes when we reverse the orientation on the surface and how it is deformed by homotheties. The following result is immediate.

\begin{proposition}\label{pr-b}
Let $\{x(s),z(s),\theta(s)\}$ be a solution of (\ref{eq1})-(\ref{eq2}).
\begin{enumerate}
\item The functions 
$$\bar{x}(s)=x(-s),\ \bar{z}(s)=z(-s),\ \bar{\theta}(s)=\theta(-s)+\pi$$
are a solution of (\ref{eq1})-(\ref{eq2}) replacing the constant $b$ by $-b$ and $\theta_0$ by $\theta_0+\pi$. 
\item If $\lambda>0$, then
$$\bar{x}(s)=\lambda x(s/\lambda),\ \bar{z}(s)=\lambda z(s/\lambda),\ \bar{\theta}(s)=\theta(s/\lambda)$$
is a solution of (\ref{eq1})  changing the constant  $b$ by $b/\lambda$ and with initial conditions 
$$\bar{x}(0)=\lambda x_0,\ \bar{z}(0)=0,\ \bar{\theta}(0)=\theta_0.$$
\end{enumerate}
\end{proposition}

Now, we prove that solutions have horizontal symmetry when their tangent vector is vertical.

\begin{proposition}[Horizontal symmetry]
\label{pr-sy}
Let $\{x(s),z(s),\theta(s)\}$ be a solution of (\ref{eq1})-(\ref{eq2}). If the tangent vector of $\gamma$ is vertical at some point $s=s_0$, then the graphic of $\gamma$ is symmetric about the horizontal  line of equation $z=z(s_0)$.
\end{proposition}
\begin{proof}
Suppose   $s_0=0$. Since $\gamma$ is vertical at $s=0$, then up to an integer multiple of $2\pi$, we find that $\theta(0)=\pi/2$ or $\theta(0)=3\pi/2$. Without loss of generality,   suppose   $\theta(0)=\pi/2$ (the other case is similar after reversing the orientation). If we define the functions 
$$\bar{x}(s)=  x(-s),\ \bar{z}(s)=  2z(0)-z(-s),\ \bar{\theta}(s)=\pi-\theta(-s),$$
then it is immediate that these functions satisfy the same equations (\ref{eq1})   with the same initial conditions at $s=0$ that $(x(s),z(s),\theta(s))$. The proof follows from the uniqueness  of ordinary differential equations.
\end{proof} 

In next proposition, we   determine  the isoparametric surfaces that satisfy (\ref{rw}).

\begin{proposition}\label{pr-circle}
Let $M$ be an isoparametric rotational surface satisfying \eqref{rw}. Then:
\begin{enumerate}
\item  $M$ is  a plane  only if $b=0$.
\item For each  $a\not=1$ and $b\neq 0$, there exists a unique round sphere satisfying (\ref{rw}). Moreover, if $b=0$, then there exists a round sphere of arbitrary radius, only if $a=1$.
\item For each $b\not=0$,  there is a circular cylinder solution of \eqref{rw}.
\end{enumerate}
\end{proposition}
\begin{proof}
The first case is clear since the principal curvatures of a plane are both zero. For a round sphere of radius $r>0$, the principal curvatures are $\kappa_1=\kappa_2=\pm1/r$ where $+$ (resp. $-$)  corresponds with the inward orientation (resp. the outward orientation). Then we ask if there is a solution $r>0$ such that $(1-a)/r=b$. If $b=0$, then $a=1$ and $r$ is arbitrary. If $b\not=0$, then necessarily $a\not=1$, and we take $r=(1-a)/b$ or $r=(a-1)/b$ in order to have $r>0$. Finally, for a circular cylinder of radius $r>0$, the principal curvatures are  $\kappa_1=0$ and $\kappa_2=\pm 1/r$ depending on the orientation. Then if $b\not=0$, we solve $a/r=b$, obtaining $r=a/b$ or $r=-a/b$ depending on the sign of $a/b$. \end{proof}

Notice that spheres cut the axis of rotation orthogonally. Moreover, all solutions that cut the axis behave in the same way.

\begin{proposition}\label{pr-or}
 If $\gamma(s)=(x(s),0,z(s))$ is a solution of (\ref{eq1}) and the graphic of $\gamma$ intersects the $z$-axis, that is, there exists $s_0>0$ such that $\lim_{s\rightarrow s_0}x(s)=0$ and $\lim_{s\rightarrow s_0}z(s)=z_1$, then 
$\gamma$ meets perpendicularly the $z$-axis. 
\end{proposition}

\begin{proof}
Firstly, we prove that $\lim_{s\rightarrow s_0}z'(s)=0$. By contradiction, suppose that $\lim_{s\rightarrow s_0}\sin\theta(s)\not=0$. Then from (\ref{eq1}) we obtain that $\lim_{s\rightarrow s_0}\theta(s)=\infty$. This implies that the curve $\gamma$ gives infinitely many turns around the point $(0,z_1)$, which is clearly a contradiction since the function $x(s)$ is positive. 
\end{proof}

Finally, in next theorem, we prove that some solutions are invariant under translations. 

\begin{theorem}[Invariance by translations] \label{th-tr}
Let  $\gamma(s)=(x(s),0,z(s))$ be a solution of (\ref{eq1}). If  the range of $\theta$  contains an interval of length $2\pi$, then  the graphic of $\gamma$ is periodic along the rotational axis and $\gamma$ is invariant by a discrete group of translations in the $z$-direction. 
\end{theorem}
\begin{proof}
Let $T>0$ be the first number such that  $\theta(T)=\theta_0+2\pi$. 

{\it Claim.} $x(T)=x(0)$.  Without loss of generality, we may assume that $\theta(0)=\theta_0=0$. Let $s_1>0$ be the first time that $\theta$ attains the value $\pi/2$. In particular, the graphic of $\gamma$ is vertical at $\gamma(s)$. It follows from  Proposition   \ref{pr-sy}  that  $x(s+s_1)=x(-s+s_1)$ and $\theta(s+s_1)=\pi-\theta(-s+s_1)$. Thus $x(2s_1)=x(0)$ and $\theta(2s_1)=\pi$. Since the range of $\theta$ contains the interval $[0,2\pi]$,  let $s_2>0$ be the first time that $\theta$ attains the value $3\pi/2$. By Proposition  \ref{pr-sy} and a similar argument as above, we find that $s_2=3s_1$ and that $4s_1$ is the first time that $\theta$ attains the value $2\pi$ with $x(4s_1)=x(0)$. In other words, $T=4s_1$ and $x(T)=0$, proving the claim. 

Once proved the claim, it is immediate that the functions 
$$\bar{x}(s)=  x(s+T),\ \bar{z}(s)=  z(s+T)-z(T),\ \bar{\theta}(s)=\theta(s+T)-2\pi$$
  satisfy   (\ref{eq1}) and with the same initial conditions at $s=0$ that $(x(s),z(s),\theta(s))$. By uniqueness, we deduce 
$$\gamma(s+T)=\gamma(s)+(0,0,z(T)),$$
proving the result. 
\end{proof} 

\section{The linear case $\kappa_1=a\kappa_2$}\label{sec4}

In this section we classify the rotational surfaces satisfying the linear relation $\kappa_1=a\kappa_2$ with $a\not=0$. Recall that $\kappa_1$ corresponds with the curvature of the profile curve and thus a circular cylinder does not satisfies the equation (\ref{rw}). The study of this class of surfaces has been done in \cite{bg} where the authors have characterized the parallels of these surfaces from a variational viewpoint and different of our Theorem \ref{1direction}.

\begin{theorem}\label{t-w1}
The rotational linear Weingarten surfaces satisfying the relation $\kappa_1=a\kappa_2$, $a\not=0$, are planes, ovaloids (including spheres) and catenoid-types. More precisely, if $\gamma$ is the generating curve, we have (see Figure \ref{figureb01}):
\begin{enumerate}
\item Case $a>0$. The curve $\gamma$ is a concave graph on the $z$-axis of a function defined on a bounded interval. The rotational surface is an ovaloid. If $a=1$, then the surface is a round sphere. 
\item Case $a<0$. The curve $\gamma$ is a convex graph on the $z$-axis. If $-1\leq a<0$, then $\gamma$ is a graph of a function defined on the entire $z$-axis  and if $a<-1$,   the function is defined on a bounded interval of the $z$-axis being asymptotic to two parallel lines. In both cases, the surface is of catenoid-type.
\end{enumerate}
\end{theorem}

\begin{proof}
 Let $\gamma(s)=(x(s),0,z(s))$ be a curve satisfying (\ref{eq1}) for $b=0$ and with initial conditions \eqref{eq2} where $x_0>0$. First, we observe that if there exists $s_0\in\r$ such that $\theta'(s_0)=0$, then $\gamma$ is a horizontal line and the surface is a plane. The proof is as follows. For the value $s=s_0$, we have $\sin\theta(s_0)=0$, that is,  $\theta(s_0)=0$ or $\theta(s_0)=\pi$ (up to a integer multiple of $2\pi$). Suppose $\theta(s_0)=0$  (a similar argument works   for $\theta(s_0)=\pi$). Then the functions 
$$\bar{x}(s)=s-s_0+x(s_0),\ \bar{z}(s)=z(s_0),\ \theta(s)=0$$
satisfy (\ref{eq1}) and with the  same  conditions at $s=s_0$ as $\{x(s),z(s),\theta(s)\}$. By uniqueness, both triple of functions agree, proving the result.
 
 Suppose $\theta'(s)\not=0$ everywhere. We discard the case $a=1$, where we know that $\gamma$ is a half-circle centered in the $z$-axis, proving the result.    We consider the initial condition $\theta_0$, which can not be $0$ because in such a case, $\theta$ is a constant function. From (\ref{eq1}), it follows that  $x''=-a(1-x'^2)/x$. After a change of variables $p=(x^2)'$, $z=x^2$ and $p=p(z)$, and some manipulations, we obtain a first integral of the above equation, obtaining
\begin{equation}\label{in2}
x'(s)^2=1+mx(s)^{2a},
\end{equation}
where $m\in\r$ is a constant of integration with $m\leq 0$. Let us observe that $m=0$ gives $x'^2=1$ everywhere, which is not possible. Thus $m<0$ and  $x(s)$ is bounded from above, namely, 
$x(s)\leq 1/(-m)^{1/(2a)}$.  Initially, let $\theta_0=\pi/2$. Then $\theta'(0)=a/x_0$ and $\theta$ is increasing (resp. decreasing) in a neighborhood of $s=0$ if $a>0$ (resp. $a<0$). In particular, the function $\theta$ is strictly monotonic.

\begin{enumerate}
\item Case $a>0$. The function $\theta$ is strictly increasing and by (\ref{eq1}), $\sin\theta(s)\not=0$ for very $s\in I$. This proves that  $\theta$ is a bounded function with $\theta(s)\leq\theta_1$, $\theta_1=\sup\theta(s)\leq\pi$.   From $x'(s)=\cos\theta(s)$, the function $x(s)$ is   decreasing   with $x'(s)\rightarrow\cos\theta_1<0$ proving that the function $x$ attains the value $x=0$ in a finite time $s_1$. As $\theta$ is a bounded monotonic function,  the third equation of (\ref{eq1}) implies  $\lim_{s\rightarrow s_1}\sin\theta(s)=0$ and thus $\theta_1=\pi$ and the intersection of $\gamma$ with the $z$-axis is orthogonal. In particular, $\lim_{s\rightarrow s_1}z'(s)=0$ and by L'H\^{o}pital rule, we have 
$$\lim_{s\rightarrow s_1}\theta'(s)=a\dfrac{\lim_{s\rightarrow s_1} \theta'(s)\cos\theta(s)}{\lim_{s\rightarrow s_1}\cos\theta(s)}=a\lim_{s\rightarrow s_1}\theta'(s).$$
As $a\not=1$, then $\lim_{s\rightarrow s_1}\theta'(s)=0$. This implies
$$\lim_{s\rightarrow s_1}z''(s)=\lim_{s\rightarrow s_1}\theta'(s)\cos\theta(s)=0.$$

By symmetry and Proposition \ref{pr-sy}, it follows that the graphic of $\gamma$ intersects the $z$-axis at two symmetric points, namely $\pm z_1$, with $z_1=\lim_{s\rightarrow s_1}z(s)$.  Since $z'(s)=\sin\theta(s)\not=0$  in $(-s_1,s_1)$, then $\gamma$ is a graph on the interval $(-z_1,z_1)$. Finally, if we write $z=z(x)$, then 
$$z''(x)=\frac{\theta'(s)}{\cos^3\theta(s)}<0,$$
proving that the graph $z=z(x)$ is concave. This means that the surface is an ovaloid.

Up to here, we have assumed $\theta_0=\pi/2$. Let us take now $\theta_0\in (0,2\pi)$, which by Proposition  \ref{pr-b}, we can suppose $\theta_0\in (0,\pi)$. 

{\it Claim.} There exists $\bar{s}$ such that $\theta(\bar{s})=\pi/2$. 

Without loss of generality,  suppose $\theta_0\in (0,\pi/2)$. The proof is by contradiction. Let $\theta_1\leq\pi/2$ such that  $\theta(s)\rightarrow\theta_1$ because $\theta$ is   monotonically increasing, in particular, $\theta'(s)\rightarrow 0$. Since $x'(s)=\cos\theta(s)$ and $x(s)$ is bounded, let $x(s)\rightarrow x_1$. By the third equation of (\ref{eq2}), we deduce  $\theta'(s)\rightarrow a\sin\theta_1/x_1>0$, a contradiction. 

Once proved the claim, after a vertical translation, we suppose $z(\bar{s})=0$. Then we use the first part of the argument   for the case $\theta_0=\pi/2$ to prove the result. 

\item Case $a<0$. We know that  $\theta$ is strictly decreasing and that   $\theta$ is a bounded function, with $\theta(s)\searrow \theta_1$ and  $\theta_1\geq 0$.   From (\ref{in2}), $x(s)$ is bounded from below with $x(s)\geq (-m)^{-1/(2a)}$. As $x'(s)\rightarrow\cos\theta_1>0$, then $x(s)\rightarrow\infty$. Moreover, this implies that $\theta'(s)$ is bounded, so the solutions of (\ref{eq2}) are defined in $\r$. As in the case $a>0$, and since $\theta'(s)\not=0$, $z''(x)>0$ and $\gamma$ is a convex graph on the $z$-axis.

If we write $z=z(x)$, it follows that 
$$z'(x)=\frac{z'(s)}{x'(s)}=\frac{\sqrt{-m}}{\sqrt{m+x^{-2a}}}.$$
Thus
$$z(x)=\sqrt{-m}\int_0^x\frac{1}{\sqrt{m+t^{-2a}}}dt.$$
This integral is of hypergeometric type and it is known (\cite{a})  that if $a\in [-1,0)$, then $\lim_{x\rightarrow\infty}z(x)=\infty$ and if $a<-1$, there exists $z_1=z_1(m,a)>0$ such that $\lim_{x\rightarrow\infty}z(x)=z_1$. Here the surface is of catenoid-type.

Suppose now that   $\theta_0\in(0,\pi)$.  Since $\theta$ is a decreasing function, if $\theta$ does not attain the value $\pi/2$, then $\theta(s)\rightarrow\theta_1$, with $\theta_1\in [\pi/2,\pi)$. Then $x(s)$ is initially decreasing. As $x'(s)\rightarrow\cos\theta_1\leq 0$, then it is not possible that  $x(s)\rightarrow\infty$, a contradiction. Once proved that the function $\theta(s)$ attains the value $\theta=\pi/2$, the argument finishes as in the case $a>0$.
\end{enumerate}
\end{proof}

The next result asserts that two rotational surfaces satisfying $\kappa_1=a\kappa_2$ are essentially unique.

\begin{corollary} Given $a\not=0$, two rotational surfaces in $\r^3$ satisfying the linear Weingarten relation $\kappa_1=a\kappa_2$ are unique up to translations and homotheties.
\end{corollary}
\begin{proof} The case $a=1$ implies that the surface is a round sphere, proving the result. Suppose $a\not=1$.   Let $M_1$ and $M_2$ be two rotational surfaces satisfying $\kappa_1=a\kappa_2$. After a translation, we suppose that the rotation axis is the same, namely, the $z$-axis. If $\gamma_i(s)=(x_i(s),0,z_i(s))$ is the generating curve of $M_i$, $i=1,2$, and after Theorem  \ref{t-w1},   the profile curves $\gamma_i$ are vertical at exactly one point. The proof finishes by using  the uniqueness of solutions of ODEs.
\end{proof}

 \begin{figure}[hbtp]
\begin{center}
\includegraphics[width=.25\textwidth]{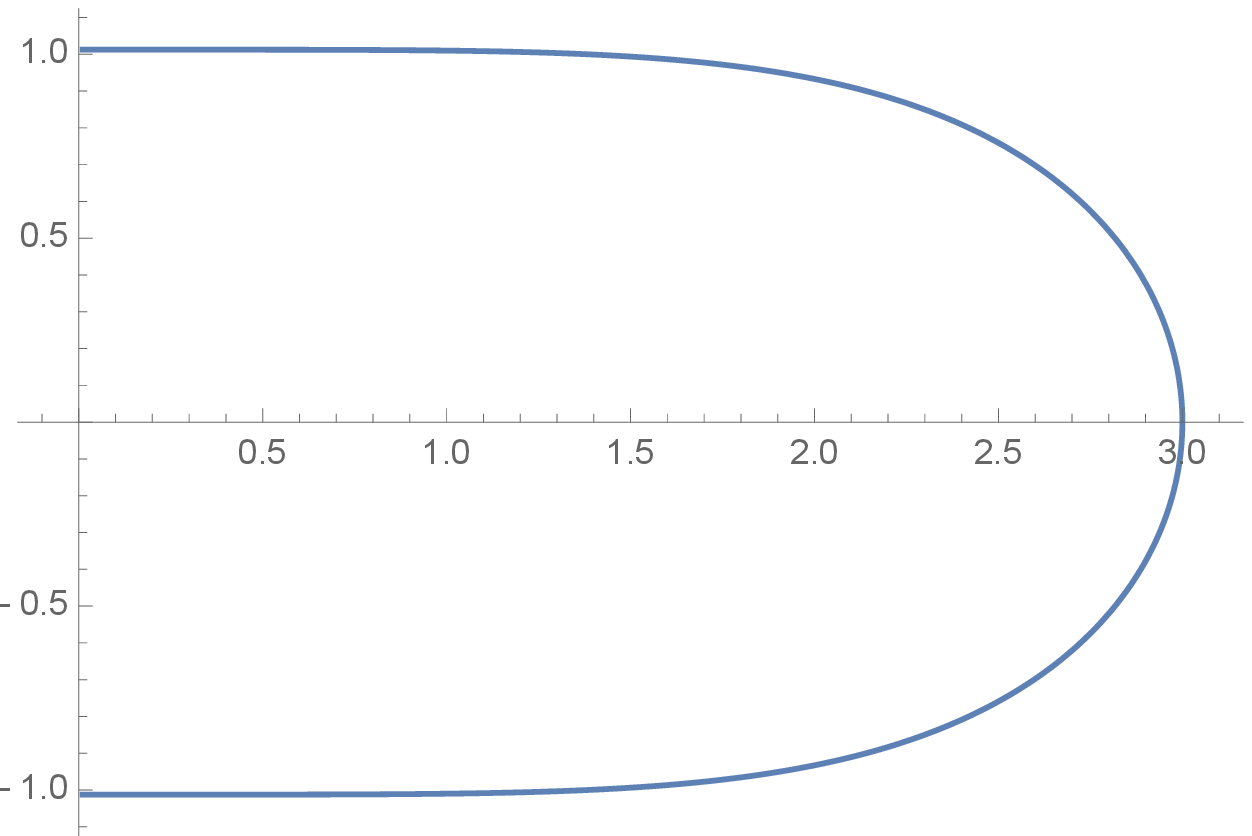}\quad   \includegraphics[width=.35\textwidth]{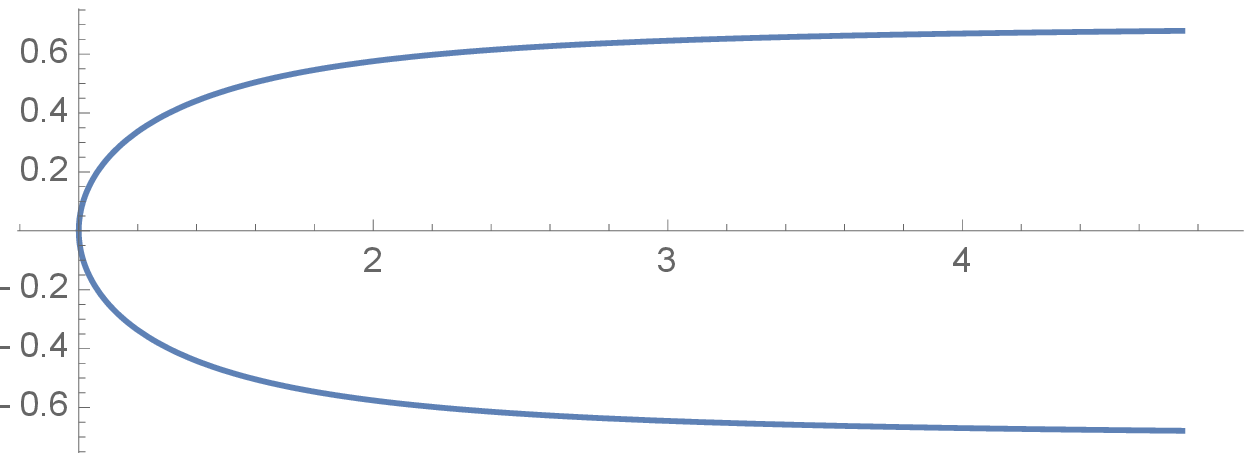}\quad  \includegraphics[width=.18\textwidth]{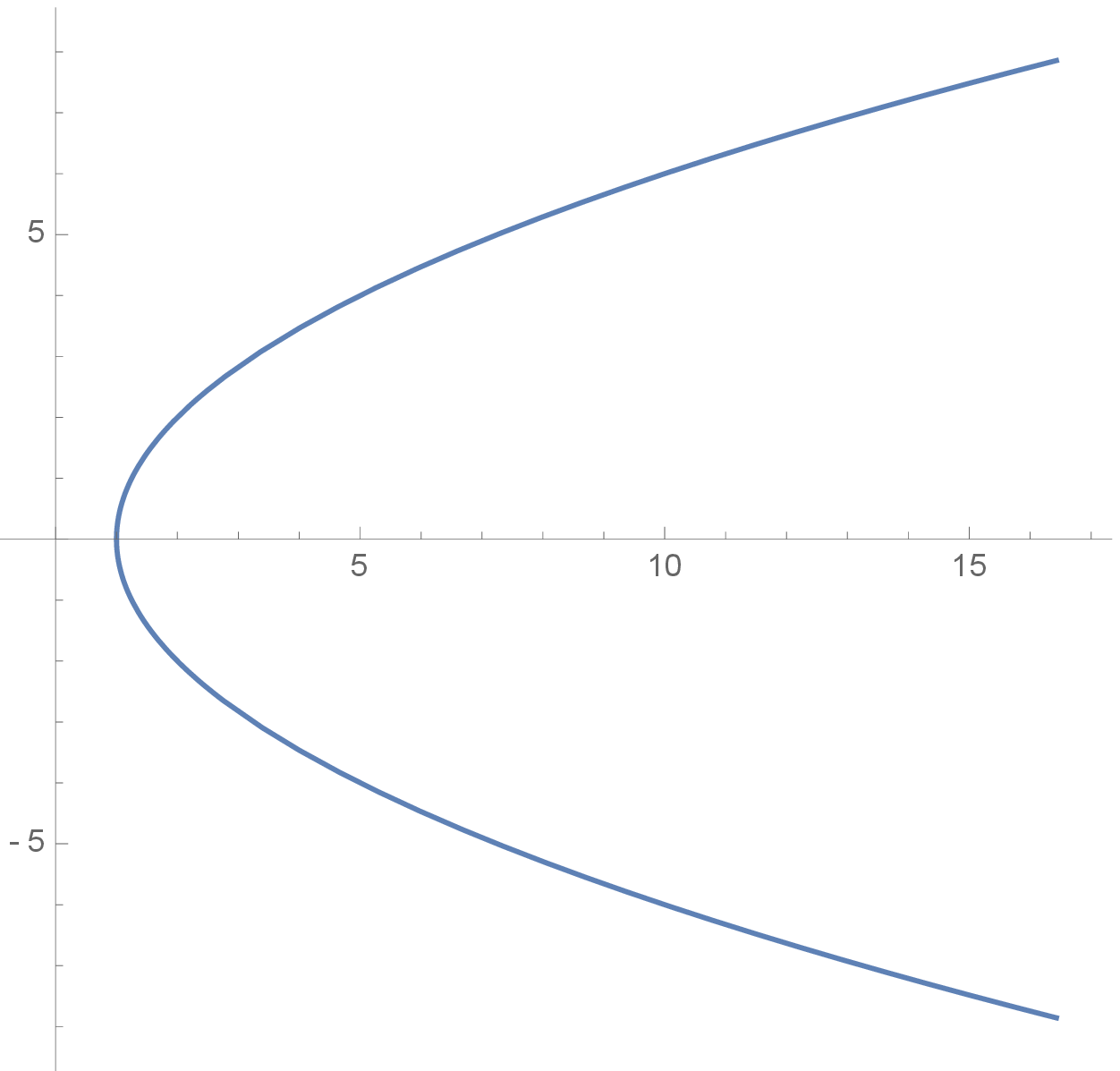}
\end{center}
\caption{Profile curves of rotational surfaces satisfying the relation $\kappa_1=a\kappa_2$. Left:  $a>0$. Middle: $a<-1$. Right: $-1\leq a<0$}\label{figureb01}
\end{figure}

\section{The general  case $\kappa_1=a\kappa_2+b$, $b\not=0$}\label{sec5}

In this section, we consider  the general case $b\not=0$ in the linear Weingarten relation $\kappa_1=a\kappa_2+b$.  The classification of the rotational linear Weingarten surfaces is given by studying the solutions of (\ref{eq1}) for all possible  values  $(x_0,\theta_0)$ on the initial conditions in (\ref{eq2}). By Proposition  \ref{pr-b}, it suffices to reduce  to the case that $b>0$ after a change $\theta_0\rightarrow\theta_0+\pi$ if necessary.     The classification will be done according to the sign of the parameter $a$. 

Firstly,  we need to give an approach  of the solutions of  the equation   (\ref{eq1})  from the viewpoint of the dynamic system theory. Here we follow a similar method as in   \cite{ch,ha} for the self-shrinker equation. We  project the vector field $(x',z',\theta')$ on the $(\theta,x)$-plane obtaining the one-parameter plane vector field
$$\left\{\begin{split}
\theta'(s)&=a\frac{\sin\theta(s)}{x(s)}+b\\
x'(s)&= \cos\theta(s).
\end{split}\right.$$
Multiplying   by $x$, which is positive, in order to eliminate the poles, we can equivalently rewrite as the autonomous system
\begin{equation}\label{a}
\left\{\begin{split}
\theta'(s)&=a \sin\theta(s) +b x(s)\\
x'(s)&=x(s) \cos\theta(s).
\end{split}\right.
\end{equation}
We study the phase plane of (\ref{a}) to visualize the trajectories of  the solutions  along the parameter $s$. By the periodicity of the trigonometric functions, we  consider the vector field
$$V(\theta,x)=(a \sin\theta +b x,x \cos\theta)$$
defined in   $[0,2\pi]\times\{x\geq 0\}$. The critical points are:
$$P_1=(0,0),\ P_2=(\pi,0),\ P_3=(\frac{\pi}{2},-\frac{a}{b}),\ P_4=(\frac{3\pi}{2},\frac{a}{b}).$$
Here the points $P_3$ and $P_4$ are regular equilibrium points whose solutions correspond with constant solutions of  (\ref{a}) and that for (\ref{eq1}) are the vertical straight lines of equations $x=\pm a/b$.  The linearization   of $V$ is 
$$LV(\theta,x)=\left(\begin{array}{cc} a\cos\theta&b\\ -x\sin\theta&\cos\theta\end{array}\right).$$
Denote $\mu_1,\mu_2$ the two eigenvalues of $LV$ at the critical points. The classification of the singularities is the following:
\begin{enumerate}
\item  The eigenvalues for  $P_1$ are $(a,1)$. If $a>0$, the eigenvalues are two real positive numbers, so $P_1$ is an unstable   node when $a\not=1$ or  an asymptotically unstable improper node if $a=1$. If $a<0$, then $P_1$ is an asymptotically unstable node.
\item The eigenvalues of $P_2$ are $(-a,-1)$. If $a>0$, $P_2$ is a stable singularity and if $a<0$, $P_2$  is an asymptotically unstable   saddle point when $a\not=-1$ or  an asymptotically unstable improper saddle point if $a=-1$.
\item The eigenvalues of $P_3$ and $P_4$ depend on the sign of $a$. If $a>0$, then the eigenvalues are $\pm\sqrt{a}$ and thus $P_3$ and $P_4$ are two asymptotically unstable saddle points. If $a<0$, the eigenvalues are $\pm\sqrt{-a}i$ and the points $P_3$ and $P_4$ are centers.
\end{enumerate}

 \subsection{Case $a>0$}
 
 In this subsection we study the case $a>0$ in (\ref{eq1}). Recall that $b>0$ by Proposition \ref{pr-b}. We consider the phase plane : see Figure \ref{figurep1}. Besides the singularities $P_1$ and $P_2$, we have $P_4=(3\pi/2,a/b)$ which is a saddle point. In order to take the initial conditions $(\theta_0,x_0)$,  we will consider those values $\theta_0$ such that the trajectories of the phase plane across some of the vertical lines $\theta=\theta_0$. In the present situation, namely, $a>0$, $b>0$,  it suffices $\theta_0=0$ and $\theta_0=3\pi/2$: see Figure \ref{figurep1}.

 \begin{figure}[hbtp]
\begin{center}
\includegraphics[width=.5\textwidth]{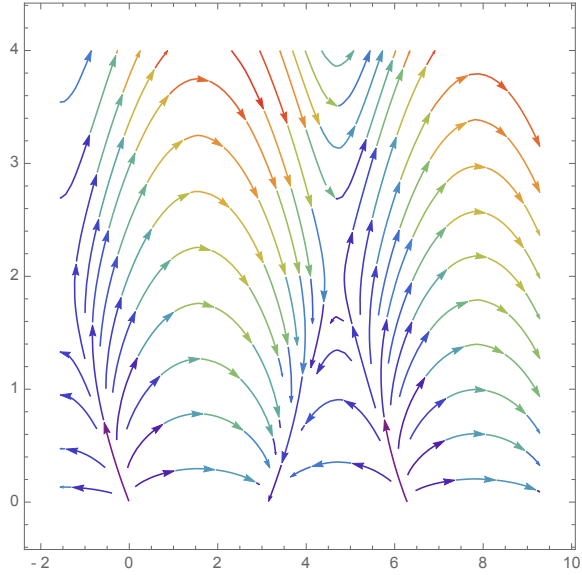}\end{center}
\caption{Phase plane for $a=2$ and  $b=1$. The point $P_4=(3\pi/2,2)$  is a saddle point. We observe that any integral curve passes through the vertical line $\theta_0=0$ or $\theta_0=3\pi/2$}\label{figurep1}
\end{figure}

 \begin{theorem}[Classification case $a>0$] Let $a>0$ and $b\not=0$. The rotational linear Weingarten surfaces are  ovaloids,   vesicle-type, pinched spheroid, immersed spheroid, cilindrical antinodoid-type, antinodoid-type and  circular cylinders.
 \end{theorem}
 
 \begin{proof}
 As mentioned above, without loss of generality, we suppose $b>0$ and we discuss (\ref{eq1})-(\ref{eq2}) for the initial conditions $\theta_0=0$ and $\theta_0=3\pi/2$.
\begin{enumerate}
\item Case $\theta_0=0$.  The generating curves appear in Figure \ref{figure6}. 

The function $\theta$ is initially increasing at $s=0$ because $\theta'(0)=b>0$. First, we point out that provided $\theta(s)\in (0,\pi)$, we have $\theta'(s)\geq b$, proving that $\theta$ crosses the value $\pi/2$ at  some time $s=s_0$ and the value $\pi$ at $s=s_1$ with   $s_1=2s_0$ and $x(s_1)=x_0$ by   Proposition  \ref{pr-sy}.

We prove now that if  $x_0$ is close to $0$ (for example, if $x_0<a/b$), then  the function $\theta$ does not attain the value $3\pi/2$. We know that after $s=s_1$, the function $x(s)$ is decreasing. If $s$ is the first time that $\theta(s)=3\pi/2$, then $\theta'(s)\geq 0$. However,  $\theta'(s)=-a/x(s)+b<-a/x(s_1)+b=-a/x_0+b<0$. This proves that $\theta$ does not attain the value $3\pi/2$. We prove that, indeed,  $\theta$ decreases after some time $s=s_3$. For this, suppose  $\theta(s)\rightarrow\theta_1$ with $\theta_1\in (\pi,3\pi/2]$. If $\theta_1=3\pi/2$,  then $\theta'(s)\rightarrow 0$ as $s\rightarrow\infty$ because the graphic of $\gamma$ can not intersect the $z$-axis. In particular, $x(s)\rightarrow\bar{x}$, with $\bar{x}\geq 0$. Then $\bar{x}\not=0$ and $\bar{x}=a/b$: a contradiction because $\bar{x}<x_0$. Thus $\theta_1\in (\pi,3\pi/2)$. In such a case, $x'(s)\rightarrow\cos\theta_1<0$, proving that $\gamma$ meets the $z$-axis, a contradiction again because it should be orthogonally (Proposition \ref{pr-or}). Definitively, $\theta$ has to attain a maximum at $s=s_3$, and then, $\theta(s)$ decreases.   If $s>s_1$ is the first time that $\theta(s)=\pi$, then $\theta'(s)\leq 0$ but now the third equation of (\ref{eq1}) gives $\theta'(s)=b>0$. This proves that $\theta$ decreases without reaching the value $\pi$. In the phase plane, the corresponding integral curves   start at the singularity $P_1$ and finish at $P_2$ in finite time. By the phase plane, we see that the trajectories when $x_0$ is close to $0$ start at the singularity $P_1$ and finish at $P_2$, which means that the curve $\gamma$ meets the $z$-axis at two points (which may coincide). Furthermore, $\gamma$ is not a graph on the $z$-axis and the surface is a vesicle.

We now increase $x_0\nearrow\infty$. By the phase plane, for large values of $x_0$, the trajectory is an entire graph on the $\theta$-axis and thus $\theta\rightarrow\infty$. By continuity, there exists $\bar{x}$ such that the integral curve with initial condition $(\theta_0=0,\bar{x})$ finishes in finite time at the singularity   $P_4=(3\pi/2,a/b)$ and by symmetry, this curve begins in $(-\pi/2,a/b)$. In particular, the branches of the solution curve $\gamma$ are asymptotic to the vertical line of equation $x=a/b$. We see that under this situation, $\gamma$ has   one self-intersection point. This is because $z'(s)<0$ for large values of $s$, that is, $z(s)$ is a decreasing function and thus $z(s)\rightarrow-\infty$ because $\gamma$ is asymptotic to a vertical cylinder. Similarly, $z(s)\rightarrow\infty$ as $s\rightarrow -\infty$ and thus the graphic of $\gamma$ has a self-intersection point. This implies that the rotational surface is a cilindrical antinodoid-type. Furthermore, if we denote $(s_1(x_0),s_2(x_0))$   the domain of $\gamma$ emphasizing the dependence on $x_0$, we have $z(s_1)\rightarrow \infty$ as $x_0\rightarrow\bar{x}$. In particular, we have surfaces of vesicle type ($z(s_2)>z(s_1)$), pinched spheroids ($z(s_1)=z(s_2)$) and immersed spheroids ($z(s_1)<z(s_2))$.  

Once $x_0$ acrosses the value $\bar{x}$, then $\theta(s)\rightarrow\infty$ and by Theorem \ref{th-tr}, the graphic of $\gamma$ is invariant by a discrete group of translations in the $z$-direction. Because as $\theta(s)\rightarrow\infty$, we have $z(s)\rightarrow-\infty$, then the surface is of antinodoid-type.

 \item Case $\theta_0=3\pi/2$. The generating curves appear in Figure \ref{figure7}. 
 
  In the phase plane (Figure \ref{figurep1}) we see that for values $x_0$ close to $0$, the trajectories go from the singularity $(2\pi,0)$ to $(\pi,0)$ indicating that the solution is a curve intersecting the $z$-axis at two points. Indeed, as $\theta'(0)=-a/x_0+b$, if $x_0$ is close to $0$, then $\theta$ is a decreasing function around $s=0$, in particular, $x(s)$ is decreasing. On the other hand, it is not possible that $\theta(s)$ attains the value $\pi$ because in such a case, $\theta'(s)\leq 0$, but (\ref{eq1}) gives $\theta'(s)=b>0$. If $s_1$ is the first point where $\theta'(s)$ vanishes, then $\theta''(s_1)=-a\sin\theta(s_1)\cos\theta(s_1)/x(s)^2<0$, a contradiction. Thus, $\theta$ is always decreasing in its domain. Definitively, $\theta(s)\rightarrow\pi$ and the corresponding trajectory finishes in finite time in the singularity $(\pi,0)$. This proves that $\gamma$ intersects orthogonally the $z$-axis, and by symmetry, the same occurs for the other branch of $\gamma$. Moreover $x'(s)=\sin\theta(s)\not=0$ which proves that $\gamma$ is a graph on the $z$-axis and because $\theta'(s)\not=0$ for all $s$, then the rotational surface is   an ovaloid. 
 
 When $x_0$ attains the value $a/b$, then we are in at equilibrium point, namely, the point $P_4$, $\gamma$ is the vertical line of equation $x=a/b$ and the surface is a circular cylinder.  
 Beyond this value for $x_0$,  the function $\theta(s)\rightarrow\infty$ by the phase plane and by Theorem \ref{th-tr}, the surface is invariant by a discrete group of vertical translations.  Now we have that as $s\rightarrow\infty$, then $\theta(s)\rightarrow \infty$ which means that $\gamma'8s)$ turn infinitely times in the counterclockwise sense. Since $z'(0)=-1$, then $z(s)\rightarrow -\infty$ as $s\rightarrow\infty$. This implies that the surface is of antinodoid-type.
   \end{enumerate}
 \end{proof}
 
 It is worth noting that in the family of ovaloids that appear in the case $\theta_0=3\pi/2$ when $x_0$ goes from $0$ to the value $a/b$, it may exist spheres among these examples. Recall that by Proposition  \ref{pr-circle}, for each pair of values $a$ and $b$ there exists a round sphere satisfying (\ref{rw}), where $r=(1-a)/b$ or $(1-1)/b$ depending on the sign of $b$ and the initial conditions (\ref{eq2}). 
  
   \begin{figure}[hbtp]
\begin{center}
\includegraphics[width=.18\textwidth]{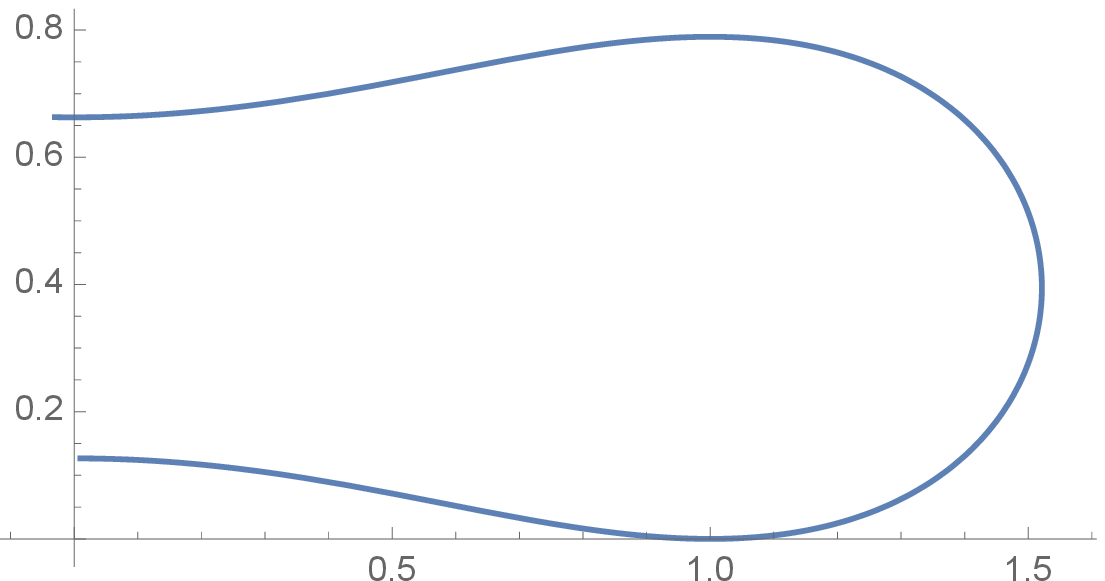} \includegraphics[width=.18\textwidth]{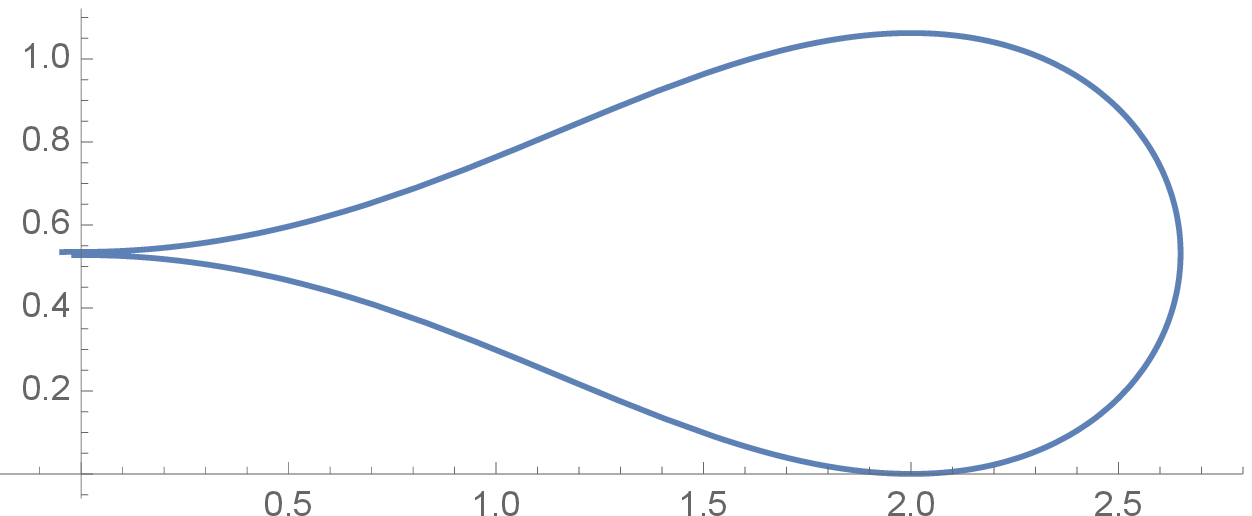} \includegraphics[width=.2\textwidth]{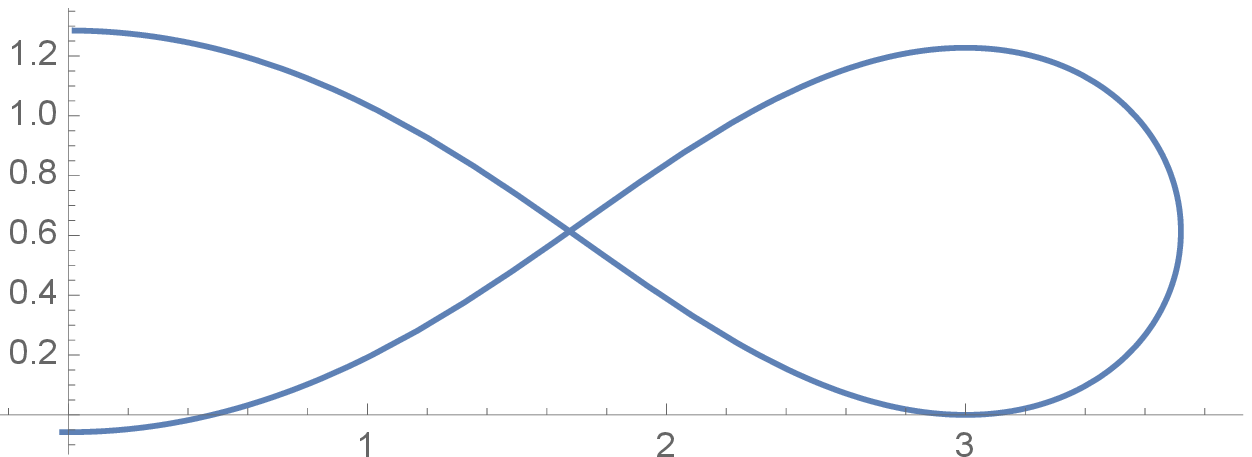} \includegraphics[width=.08\textwidth]{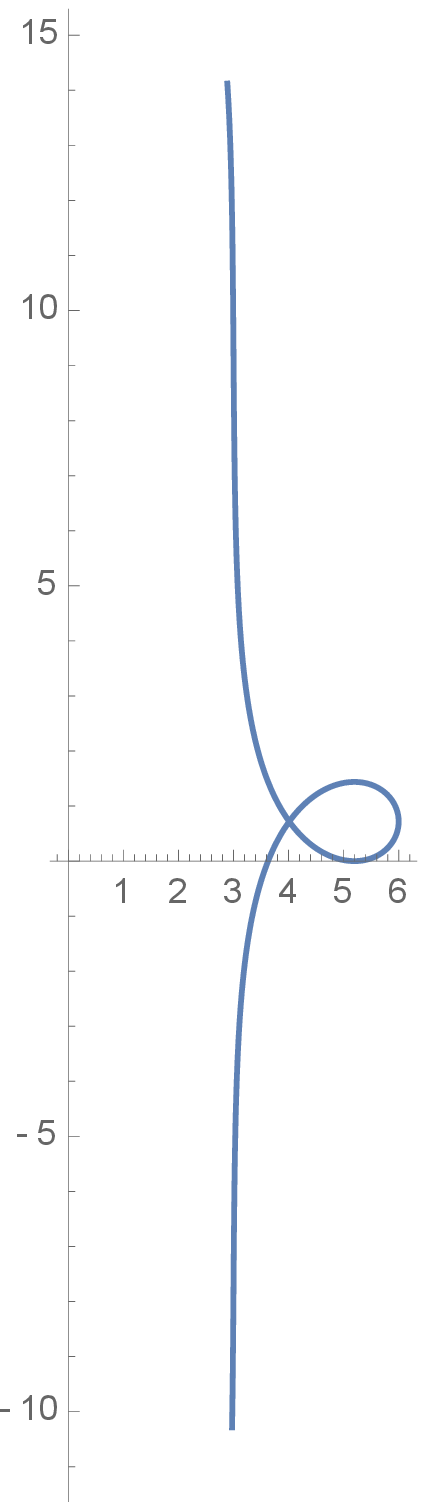} \includegraphics[width=.18\textwidth]{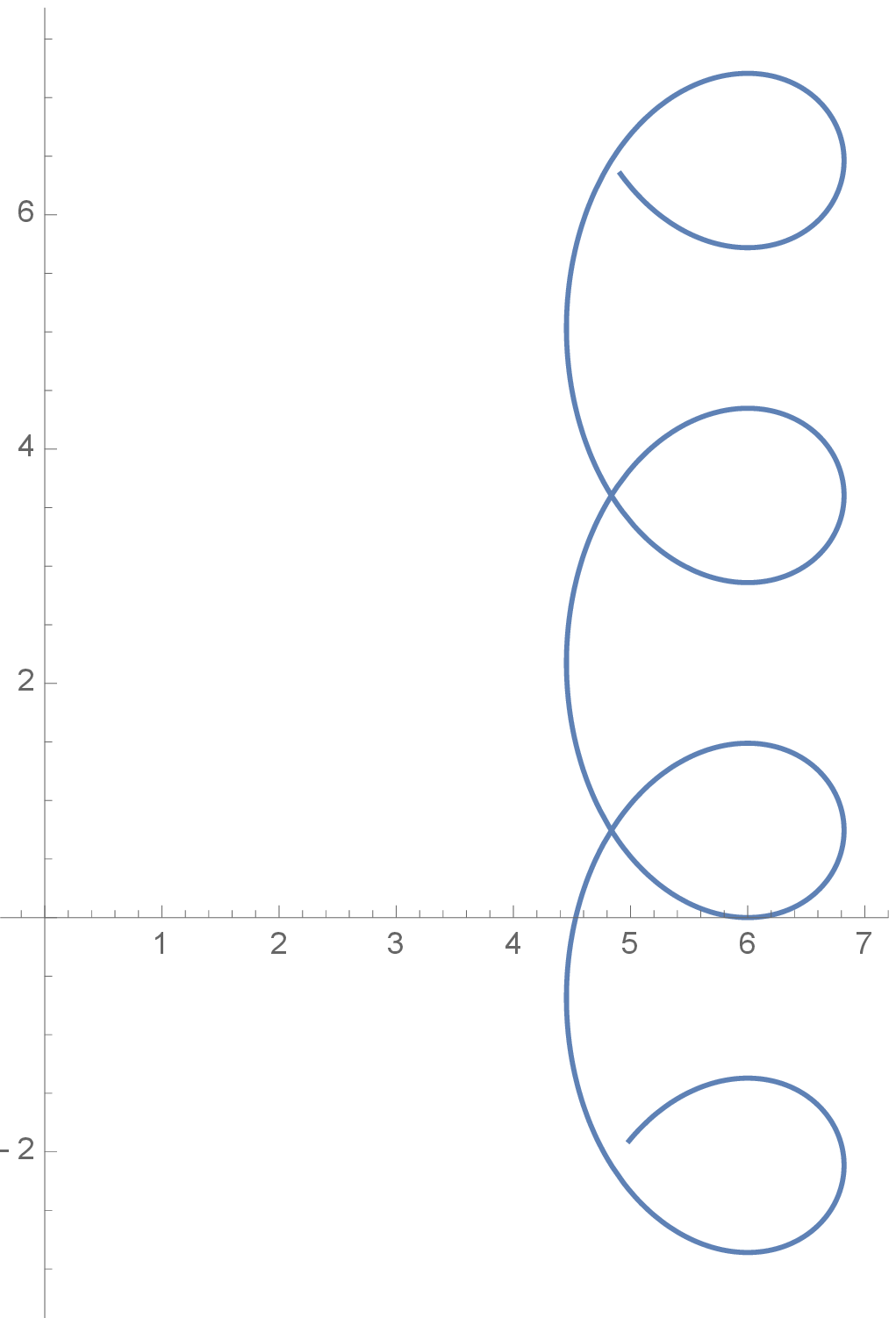}   
\end{center}
\caption{Case $a>0$, $b>0$ and $\theta_0=0$. Here $a=3$, $b=1$.  From left to right, the values of $x_0$ are: $1$, $2$, $3$, $\simeq 5.196$ and $6$, respectively}\label{figure6}
\end{figure}

    \begin{figure}[hbtp]
\begin{center}
\includegraphics[width=.2\textwidth]{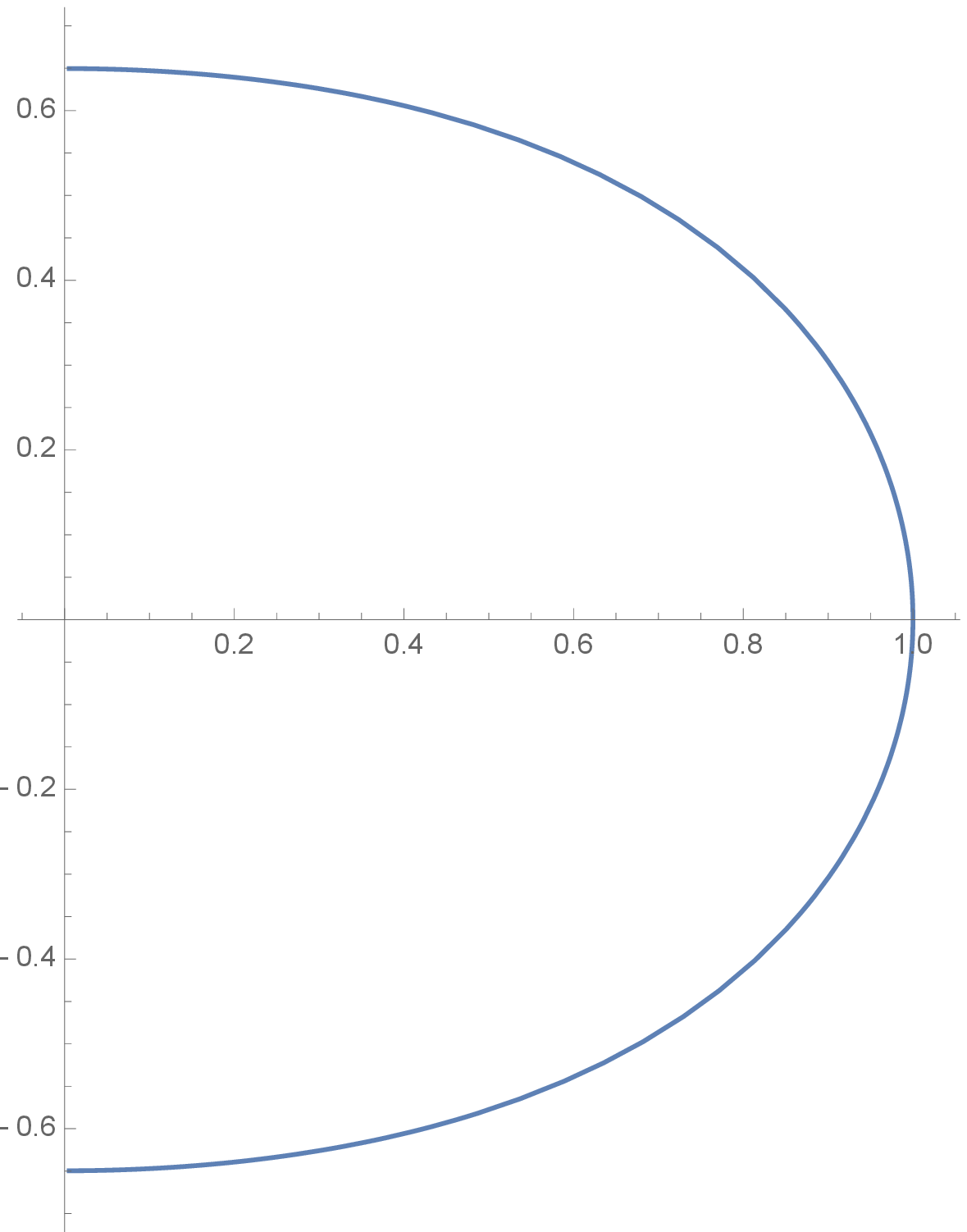} \includegraphics[width=.15\textwidth]{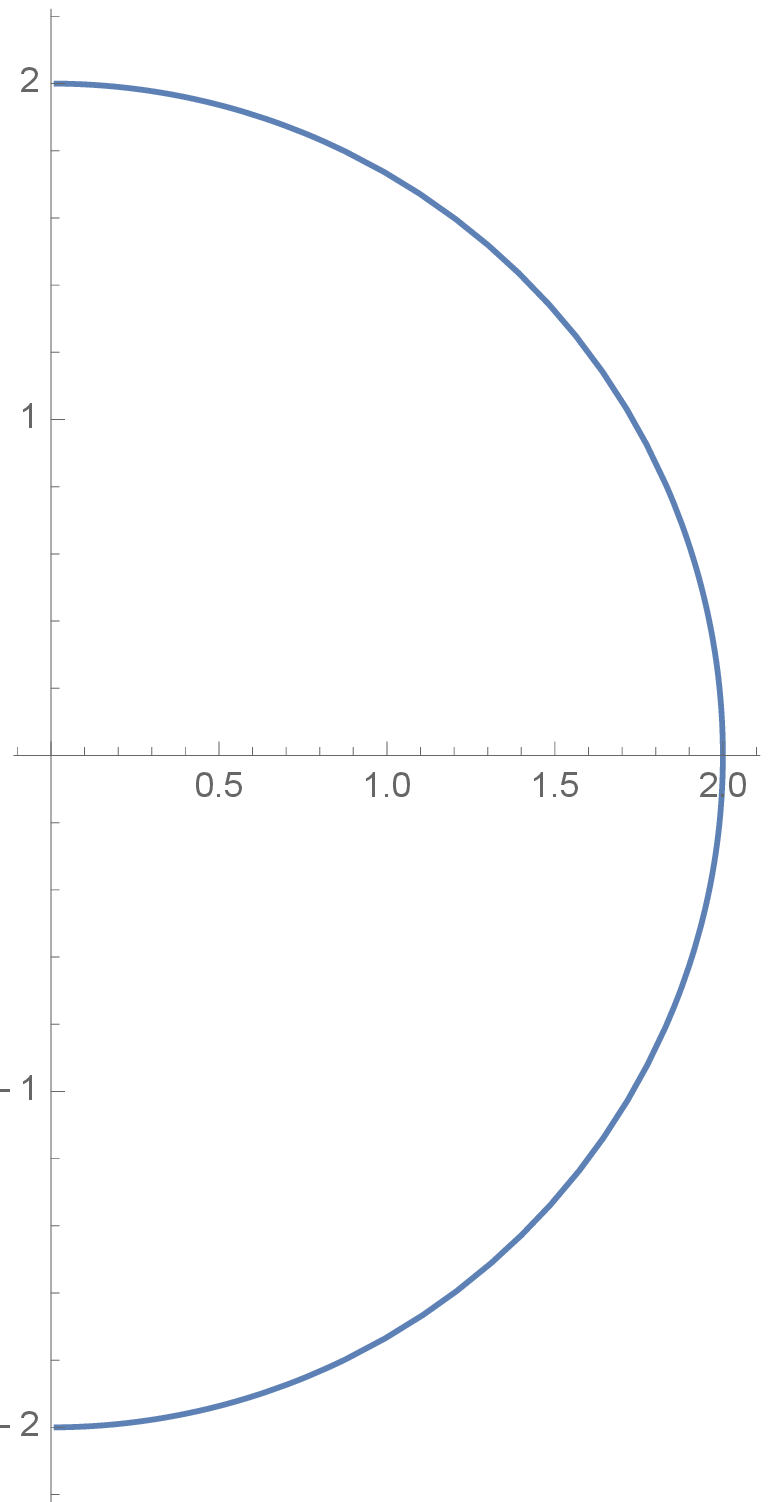} \includegraphics[width=.25\textwidth]{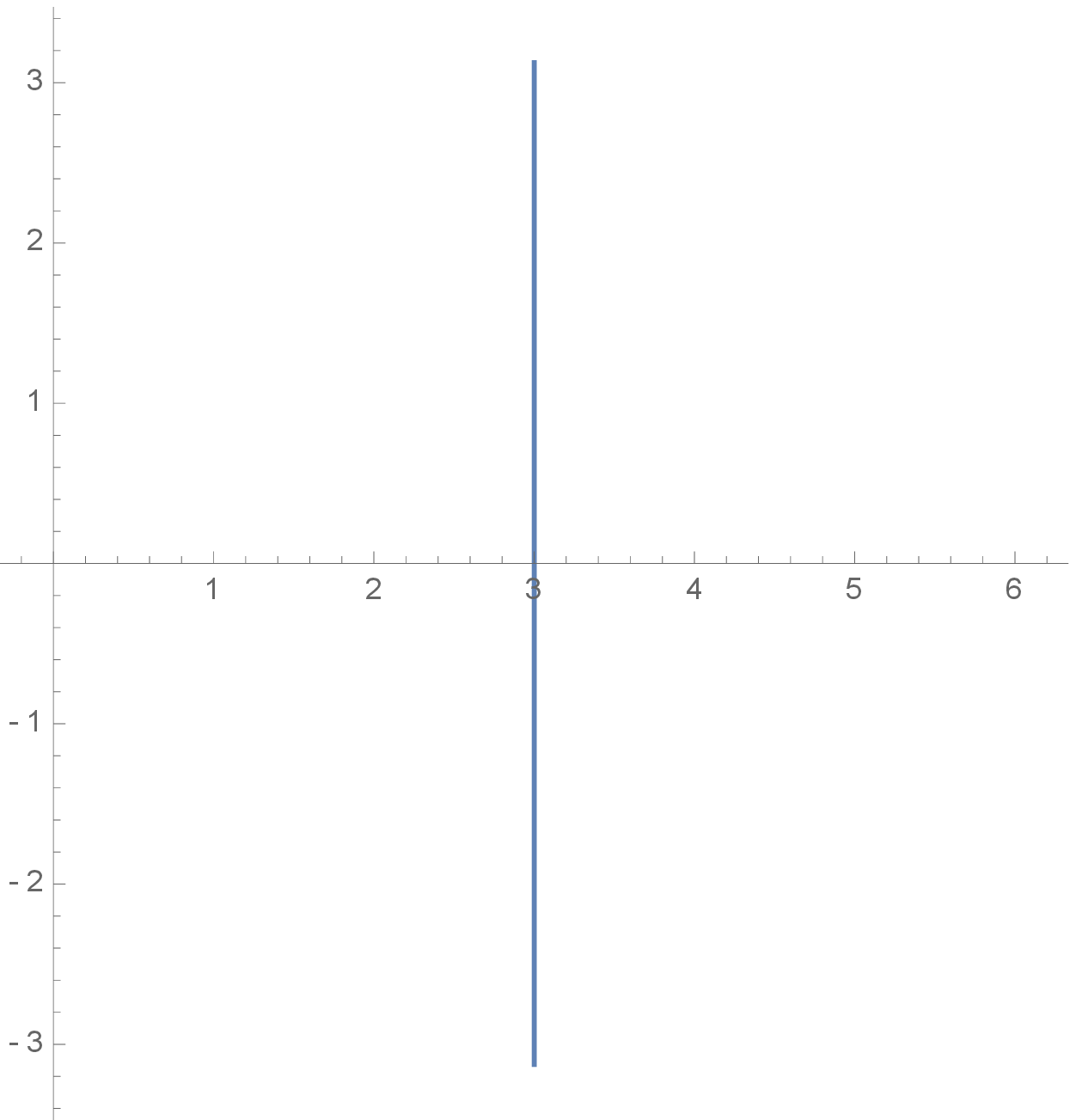} \includegraphics[width=.16\textwidth]{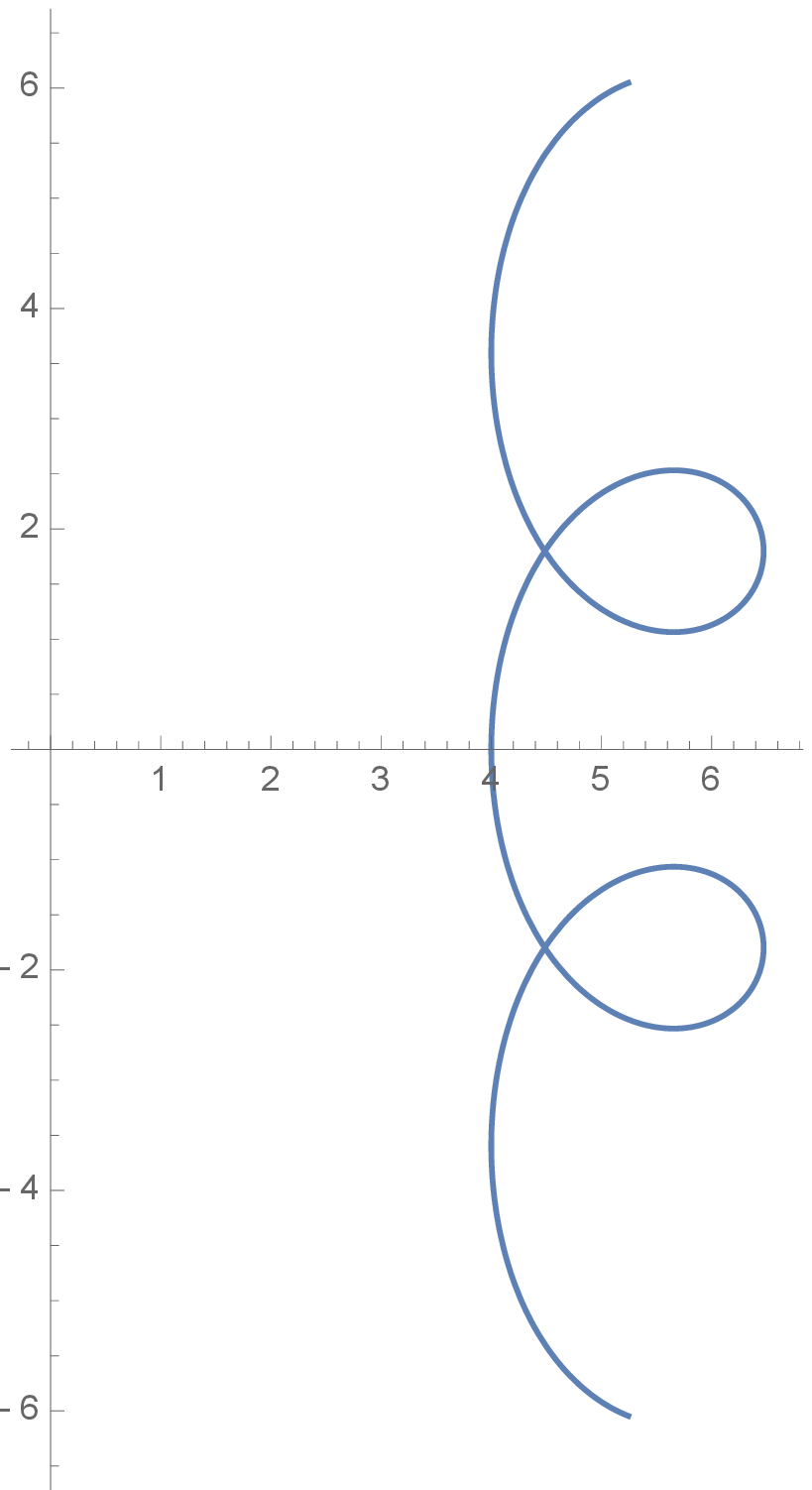}    
\end{center}
\caption{Case $a>0$, $b>0$ and $\theta_0=3\pi/2$. Here $a=3$, $b=1$.  From left to right, the values of $x_0$ are: $1$, $2$, $3$, $4$, respectively}\label{figure7}
\end{figure}

 \subsection{Case $a<0$}

In this subsection we consider $a<0$ in the Weingarten relation (\ref{rw}). Recall $b>0$ by Proposition \ref{pr-b}. As it was pointed in the introduction, many of the surfaces that we will obtain share similar properties with Delaunay surfaces ($a=-1$). We will see in Theorem \ref{t53} below that, in contrast to the case $a>0$,  the only rotational surfaces intersecting the axis are spheres.  
 
We depict in Figure \ref{figurep} the phase plane for the case $a<0$. Now we have that $P_3=(\pi/2,-a/b)$ is a center.   By the phase plane again, it suffices  to consider   $\theta_0=\pi/2$ in (\ref{eq2}) because all trajectories cross the vertical line $\theta_0=\pi/2$.

\begin{figure}[hbtp]
\begin{center}
\includegraphics[width=.4\textwidth]{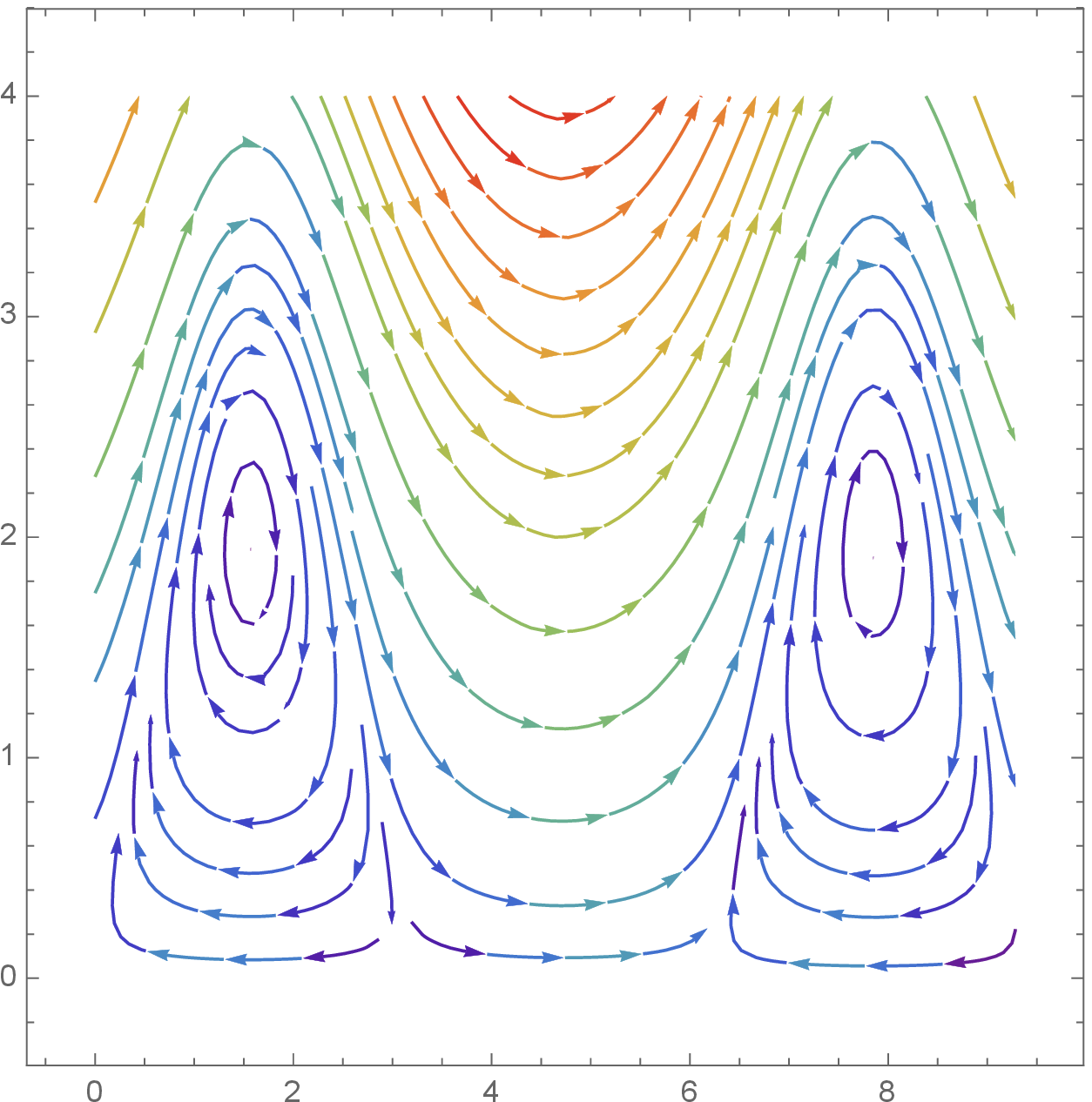}\end{center}
\caption{Phase plane for $a=-2$ and $b=1$. The point $P_3=(\pi/2,2)$  is a center}\label{figurep}
\end{figure}

\begin{theorem}\label{t53}
Let $a<0$ and $b\not=0$. The rotational linear Weingarten surfaces   are  unduloid-type, circular cylinders, spheres and nodoid-type.
\end{theorem}

\begin{proof}

 We see by the phase plane that if $x_0$ is close to $0$, the integral curve through $(\pi/2,x_0)$ is a cycle around $P_3$ which means that the angle function $\theta(s)$ varies in a bounded interval of length less than $2\pi$. As $x_0$ increases and arrives to $x_0=-a/b$, then we know that it is an equilibrium point and the solution is the vertical line of equation $x=-a/b$. After this value, the function $\theta$ follows being bounded until a critical value for $x_0$ where beyond the integral curve is defined in the entire $\theta$-axis. This implies that $\theta$ goes to $\infty$ and the velocity vector $\gamma'(s)$ turns infinitely times. We give the details.

At $s=0$ we have $\theta'(0)=a/x_0+b$. Since $a<0$, if $x_0$ is sufficiently close to $0$, then $\theta'(0)<0$ and if $x_0$ is sufficiently big, then $\theta'(0)>0$. In fact, if $x_0=-a/b$ in (\ref{eq2}), it is immediate that the solution of (\ref{eq1}) is a vertical line and the corresponding surface is a vertical circular cylinder. Suppose that $x_0<-a/b$. Then the function $\theta$ is decreasing at $s=0$. It is not possible that $\theta$ attains the value $0$ because at the first such a point $s$, we have $\theta'(s)\leq 0$, but from equation (\ref{eq1}), $\theta'(s)=b>0$. Thus $\theta$ is bounded from below by some $\theta_1$ with $\theta_1\geq 0$, which we may assume is its infimum. 
 
We claim that $\theta_1$ is a minimum of the function $\theta(s)$. On the contrary, $\theta$ is a decreasing function. If $\theta_1>0$, then $x(s)\rightarrow\infty$ and $\theta'(s)\rightarrow 0$, which is not possible by the third equation of (\ref{eq1}) which leads to 
$\theta'(s)\rightarrow b$. If $\theta_1=0$ and independently if $x(s)\rightarrow \infty$ or $x(s)\rightarrow \bar{x}$, for some $\bar{x}>0$,   we obtain the same contradiction.  
 
By the claim, there exists $s_0$ such that $\theta'(s_0)=0$ and since $\theta''(s_0)>0$, $s_0$ is a minimum of $\theta$. Now the function $\theta$ increases after $s=s_0$. We prove that $\theta$ crosses the value $\pi/2$. On the contrary, $\theta(s)$ is bounded from above by some value $\theta_2\leq\pi/2$. It is not possible that $\theta'$ vanishes at some point $s$ because at $s=s_1$ we have
$$\theta''(s_1)=-a\frac{\sin\theta(s_1)\cos\theta(s_1)}{x^2(s_1)}.
$$
 and  we infer that  $\theta''(s)>0$ so $s$ would be a minimum. Thus $\theta$ is strictly increasing. The arguments are now known. As $x'(s)\rightarrow\cos\theta_2$, if $\theta_2<\pi/2$, then $x(s)\rightarrow\infty$ and $\theta'(s)\rightarrow b\not=0$, a contradiction. If $\theta_2=\pi/2$ and $x(s)\rightarrow\infty$, we arrive to the same contradiction. If $x(s)\rightarrow \bar{x}$, for some $\bar{x}>0$, and as $\theta'(s)\rightarrow 0$, then $x_1=-a/b$. This proves that the trajectories  of the phase plane arrive to the point $P_3$: a contradiction, because $P_3$ is a center. 
 
Let $s_1>0$ be the first  time where $\theta(s_1)=\pi/2$.  The proof  finishes using Proposition  \ref{pr-sy} where the embedded graphic $\gamma([0,s_1])$ reflects about the horizontal line $z=z(s_1)$. This proves that    the graphic of $\gamma$ is embedded and periodic in the $z$-direction with  period $T=2s_1$. The surface is of unduloid-type.
 
If $x_0=-a/b$, we know that the solution is the vertical line and the surface is a circular cylinder. 
 
Suppose $x_0>-a/b$ and close to $-a/b$, we know by the phase plane that the trajectories are closed round the center $P_3$. This proves that the angle $\theta(s)$ and the function $x(s)$ are bounded in some interval. Thus  $\theta$ oscillates around $\pi/2$ obtaining the surface is of unduloid-type. This occurs until a certain value $x_0=\bar{x}_0$ which is the last time that $\gamma$ leaves to be of nodoid-type and if $x_0>\bar{x}_0$, then the trajectories in the phase plane are of infinite length. In fact we know by Proposition  \ref{pr-circle} that this occurs when $\gamma$ is a half-circle (the length of variation of $\theta$ is exactly $\pi$): here $\bar{x}_0=(1-a)/b$ because $\theta$ is increasing at $s=0$. In the phase plane, this solution corresponds with the trajectory starting at $(\pi/2,\bar{x})$ and finishes in finite time $s=s_0$ at the singularity $P_2=(\pi,0)$. The other branch of this trajectory finishes in $P_1=(0,0)$.
 
Finally, when $x_0>\bar{x}_0$,   the phase plane implies that the trajectories are entire graphs on the $\theta$-axis and   increasing   with $\theta(s)\rightarrow\infty$. Using Proposition  \ref{pr-sy} and   Theorem  \ref{th-tr}, the graphic of the solution is a periodic curve with infinite self-intersections. Now we have $z'(0)=1$, and thus, $z$ is increasing at $s=0$. Since we know that $x$ is not a bounded function, then we deduce that $z(s)\rightarrow\infty$ as $s\rightarrow\infty$.   This proves that the surface is of nodoid-type. 
 \end{proof}

 In Figure \ref{anegativo} all types of surfaces in the case $a<0$ and $b>0$ appear. We point out that when $x_0\rightarrow 0$, the unduloid-type solution degenerate in a sequence of tangent spheres centered at the $z$-axis.

\begin{figure}[hbtp]
\begin{center}
 \includegraphics[width=.1\textwidth]{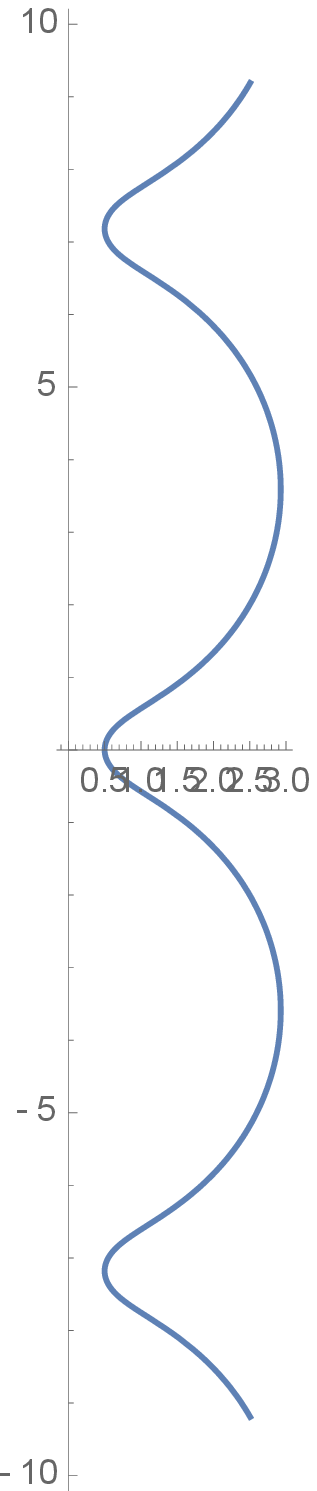} \includegraphics[width=.15\textwidth]{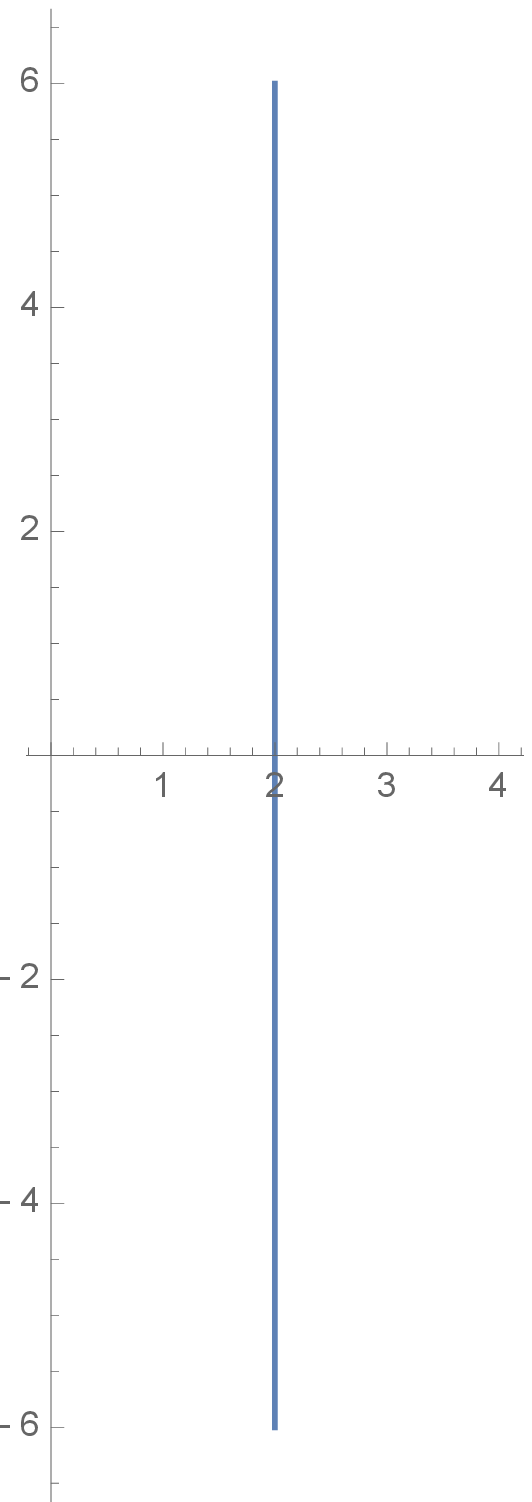} \includegraphics[width=.1\textwidth]{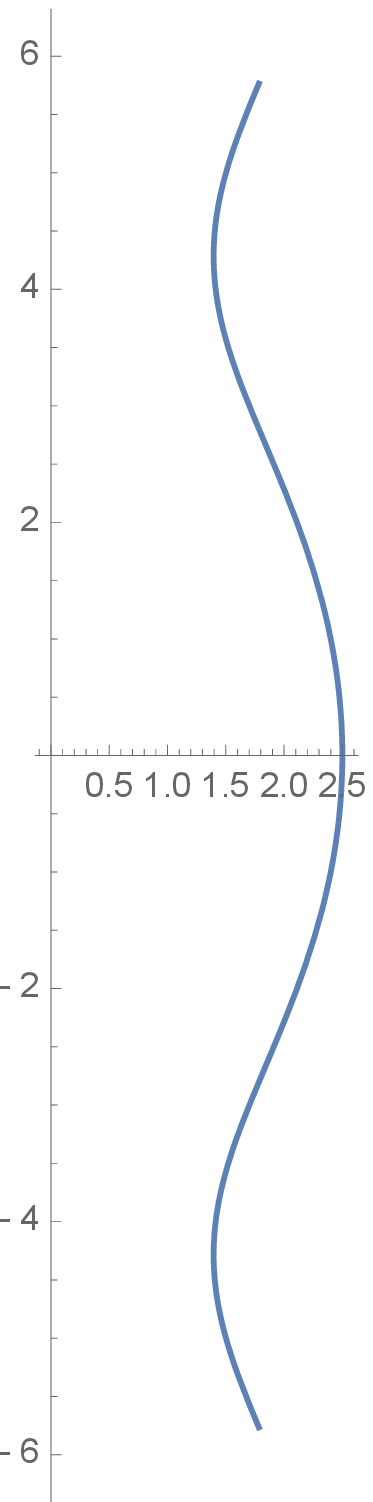} \includegraphics[width=.2\textwidth]{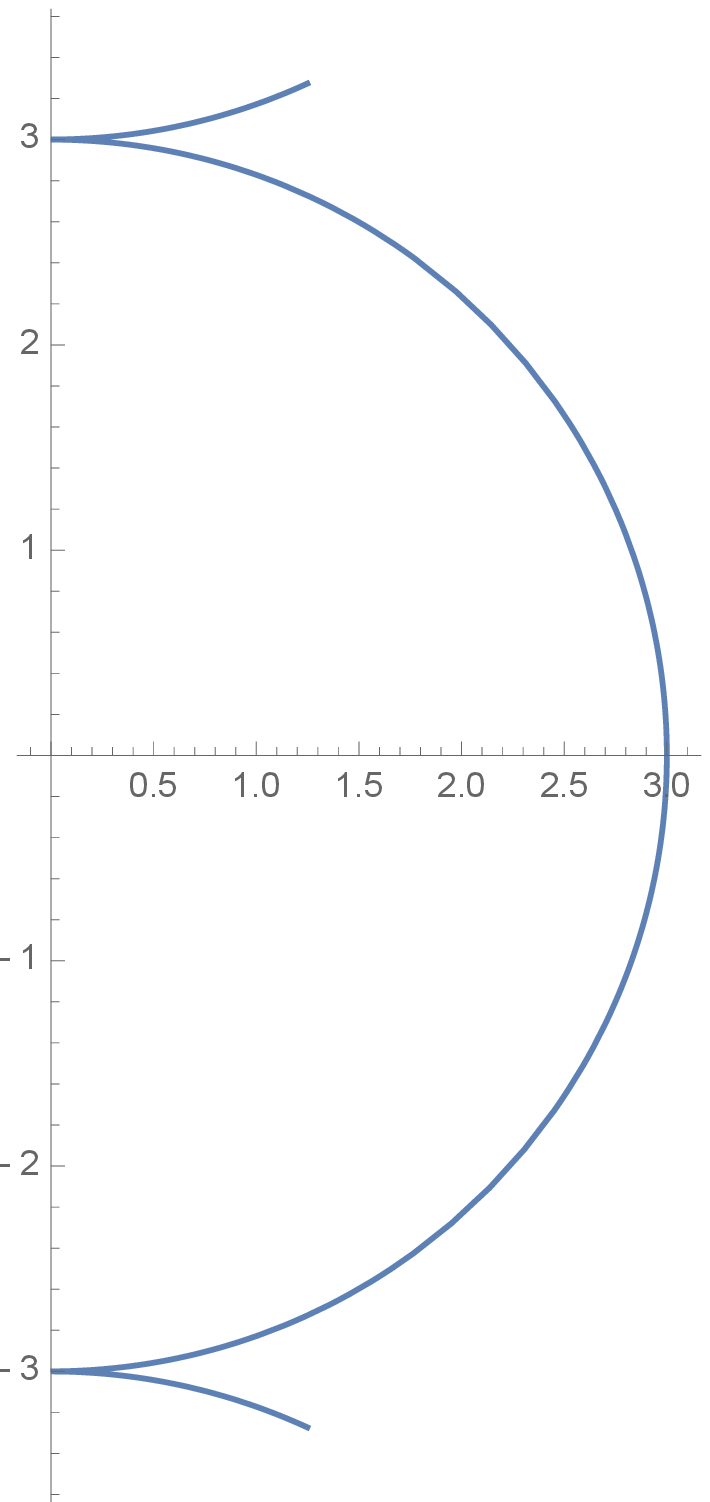}  \includegraphics[width=.2\textwidth]{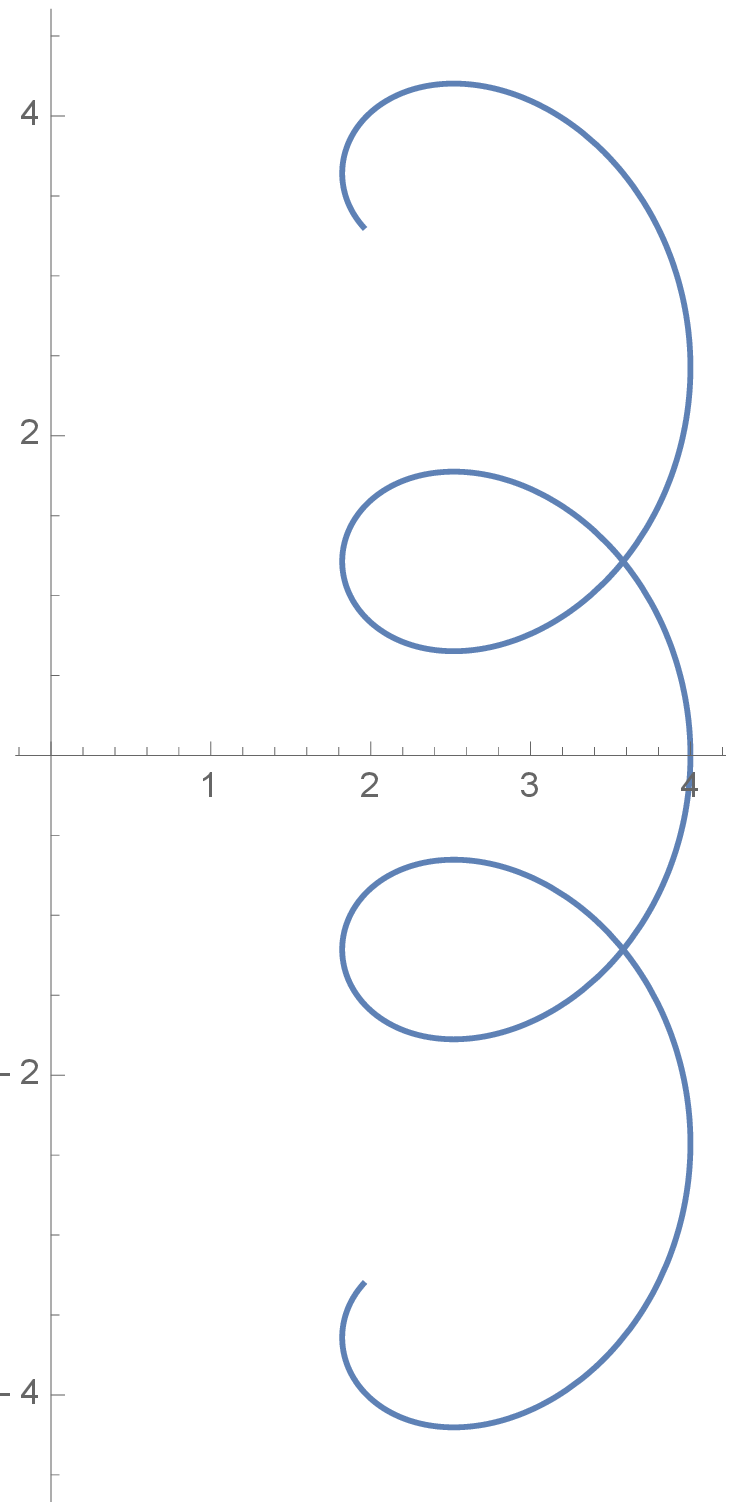} 
\end{center}
\caption{Case $a<0$, $b>0$ and $\theta_0=\pi/2$. Here $a=-2$, $b=1$. From left to right, the values of $x_0$ are:  $0.5$, $2$, $2.5$, $3$ and $4$, respectively}\label{anegativo}
\end{figure}
  

\end{document}